\documentclass[11pt]{amsart}

\usepackage{amsmath, amssymb, amscd, amsthm,epsfig,epic,eepic}
\usepackage[justify]{ragged2e}

\usepackage[T1]{fontenc}
\usepackage{framed}
\usepackage{multicol}
\usepackage[normalem]{ulem}
\usepackage{booktabs}
\usepackage{multirow}
\usepackage{bbding}
\usepackage{mathrsfs}
\usepackage{mathdots}
\usepackage[shortlabels]{enumitem}
\usepackage[new]{old-arrows}

\usepackage[backend=bibtex,style=alphabetic]{biblatex}
\makeatletter
     \let\oldfootnote\footnote
     \def\footnote{\@ifstar\footnote@star\footnote@nostar}
     \def\footnote@star#1{{\let\thefootnote\relax\footnotetext{#1}}}
     \def\footnote@nostar{\oldfootnote}
     \makeatother

\newcommand{\Z}{\mathbf{Z}}%{\numberset{Z}}
\newcommand{\Q}{\mathbf{Q}}%{\numberset{Q}}
\newcommand{\C}{\mathbf{C}}%{\numberset{C}}
\renewcommand{\AA}{\mathbf{A}}

\makeatletter
\newcommand*{\defeq}{\mathrel{\rlap{%
                      \raisebox{0.3ex}{$\cdot$}}%
                      \raisebox{-0.3ex}{$\cdot$}}%
                      =}
\makeatother

\newcommand{\proj}{\mathbf{P}}
\newcommand{\Gm}{\mathbf{G}_m}

\newcommand{\Ga}{\mathbf{G}_a}

\DeclareMathOperator{\Hom}{Hom}
\DeclareMathOperator{\Aut}{Aut}
\DeclareMathOperator{\Bir}{Bir}
\DeclareMathOperator{\Spec}{Spec}
\DeclareMathOperator{\Gal}{Gal}
\DeclareMathOperator{\Frac}{Frac}

\DeclareMathOperator{\Stab}{Stab}

\DeclareMathOperator{\GL}{GL}

\DeclareMathOperator{\PGL}{PGL}

\DeclareMathOperator{\Pic}{Pic}

\DeclareMathOperator{\rank}{rank}

\usepackage{xcolor, color}
\usepackage{hyperref, cleveref}
\hypersetup{
    colorlinks=true,
    linkcolor=teal,
    citecolor=purple,
    filecolor=purple,
    urlcolor=purple,
    linktocpage=true
}
\usepackage{comment}
\usepackage[inner=2.75cm,outer=2.75cm,top=2.75cm,bottom=2.5cm]{geometry}

\usepackage[all]{xy}
\usepackage{tikz-cd}

\theoremstyle{plain}

\newtheorem{theorem}{Theorem}[section]
\newtheorem{lemma}[theorem]{Lemma}
\newtheorem{proposition}[theorem]{Proposition}
\newtheorem{corollary}[theorem]{Corollary}
\newtheorem{theorem*}{Theorem}
\newtheorem{proposition*}[theorem*]{Proposition}
\newtheorem{lemma*}[theorem*]{Lemma}
\newtheorem*{proposition**}{Proposition}
\newtheorem{comment*}{Comment}
\newtheorem{theoremA}{Theorem}

\newtheorem{propositionA}[theoremA]{Proposition}

\newtheorem{corollaryA}[theoremA]{Corollary}

\newtheorem*{Acknowledgements}{Acknowledgements} 

\theoremstyle{definition}
\newtheorem{definition}[theorem]{Definition}
\newtheorem{remark}[theorem]{Remark}

\newtheorem{example}[theorem]{Example}
\numberwithin{equation}{section}
\newtheorem*{conventions*}{Conventions}

\begin{document}

\title{Isotrivial elliptic surfaces in positive characteristic}

\author{Pascal Fong, Matilde Maccan}
	\address{Leibniz Universität Hannover, Institut für Algebraische Geometrie, Welfengarten 1, 30167 Hannover}
	\email{fong@math.uni-hannover.de}
	\address{Faculty of Mathematics, Ruhr University Bochum, Universitätsstraße 150, 44780 Bochum, Germany}
	\email{matilde.maccan@ruhr-uni-bochum.de} 

\footnote*{Keywords: elliptic surfaces, diagonalizable group schemes, finite group schemes, positive characteristic.}
\footnote*{2020 Mathematics Subject Classification:
Primary: 14J27, Secondary: 14L15, 14L30.}

\begin{abstract}
We study relatively minimal surfaces equipped with a strongly isotrivial elliptic fibration in positive characteristic by means of the notion of equivariantly normal curves introduced and developed recently by Brion in \cite{Brion2,diagos}. Such surfaces are isomorphic to a contracted product $E\times^G X$, where $E$ is an elliptic curve, $G$ is a finite subgroup scheme of $E$ and $X$ is a $G$-normal curve. Using this description, we compute their Betti numbers to determine their birational classes. This allows us to complete the classification of maximal automorphism groups of surfaces in any characteristic, extending the result in characteristic zero obtained in \cite{fong_1}. When $G$ is diagonalizable, we compute additional invariants to study the structure of their Picard schemes.
\end{abstract}

\maketitle
%-------------------------------------------------------------------------------

\section*{Introduction}

In a pioneer article \cite{Kodaira}, Kodaira classified the singular fibers of \emph{elliptic surfaces}, i.e., smooth projective surfaces equipped with an elliptic fibration. Most of those surfaces are of Kodaira dimension one, but may also be of negative or zero dimension. Their classification in positive characteristic, up to birational transformations, has been provided in \cite{BombieriMumford,BombieriMumfordIII}. The last reference focuses on characteristics two and three, in which the family of quasi-hyperelliptic surfaces appears.

\medskip

In this article, we study a family of smooth projective elliptic surfaces, over an algebraically closed field of positive characteristic, defined through the notion of \emph{equivariantly normal curves} introduced and developed recently by Brion in \cite{Brion2,diagos}. This tool gives us a new approach to study some elliptic surfaces.

\medskip

The construction goes as follows: from a geometric point of view, we start with a relatively minimal smooth projective surface $S$, equipped with a \emph{strongly isotrivial elliptic fibration}. To the extent of the authors' knowledge, this notion is new, and we mean that the generic fiber of the Jacobian is a base change of an elliptic curve over $k$ (see \Cref{def:stronglyiso}). To make this more concrete in terms of algebraic groups, we show that this is actually equivalent for $S$ to be equipped with a faithful action of an elliptic curve $E$. Moreover, it is also equivalent to $S$ being isomorphic to a \emph{contracted product}
\[
E \times^G X \defeq (E \times X)/G,
\]
for a diagonal action of $G$, where $G$ is a finite subgroup scheme of $E$ and $X$ is a $G$-normal projective curve. By that, we mean that every finite birational $G$-morphism with target $X$ is an isomorphism; this notion is a remedy to the fact that in positive characteristic the action does not lift to the normalization in general. Notice that in general, the triplet $(E,G,X)$ is not uniquely determined. However, if $S$ is not an abelian surface, then the elliptic curve $E$ is the largest abelian subvariety of the automorphism group $\Aut(S)$. If moreover $S$ has Picard rank two, then the fibers of the two contraction morphisms with source $S$ uniquely determine the finite subgroup scheme $G\subset E$ and the $G$-normal curve $X$.

\medskip

We aim to study the surfaces of this shape and determine some of their invariants. Those surfaces are always relatively minimal, that means they do not contain any $(-1)$-curves. Notice that this family of surfaces contains all the hyperelliptic and quasi-hyperelliptic surfaces. In characteristic zero, the analogous construction provides families of surfaces that have been studied intensively: in that case, the group $G$ is always constant and the curve $X$ is always smooth. Quotients of products of curves by finite groups is a classical topic, see e.g.\ \cite{Bauer_Catanese_Grunewald,Bauer_Catanese_Grunewald_Pignatelli,Dedieu_Perroni,Kentaro,Lorenzini}. In the present work, we introduce a variant of this well-known approach, by considering quotients by \emph{nonreduced finite group schemes} in positive characteristic, which yields new examples and a different geometric behavior. A crucial point is that the curve $ Y \defeq X/G$ is smooth and the elliptic fibration $S\to Y$ does not admit a section in general.

\medskip

From now on, we fix an algebraically closed field $k$ of characteristic $p>0$.

The first main result is that, using the explicit description of $S$ as a contracted product, one can determine its Betti numbers. This allows us to locate our surfaces in the classification of surfaces in positive characteristic.

\begin{theoremA}\label{TheoremA}

Let $E$ be an elliptic curve, $G$ a finite subgroup scheme of $E$ acting on a $G$-normal curve $X$ with quotient $Y\defeq X/G$. We denote by $g(Y)$ the genus of the smooth projective curve $Y$ and $S \defeq E\times^G X$. Then $S$ is a relatively minimal surface such that $\kappa(S) = \kappa'(X)$, where $\kappa'(X)$ denotes the Iitaka dimension of the dualizing sheaf $\omega_X$, and has the following Betti numbers:

\begin{itemize}
    \item $b_1(S) = 2 + 2g(Y)$,
    \item $b_2(S) = 2 + 4g(Y)$.
\end{itemize}

Moreover, the triples $(S,X,Y)$ are classified as follows:

\smallskip
\begin{center}
    \begin{tabular}{| c | c | c | c | c | c | c |}
		
        \hline
  
		$\kappa(S)$ & $S$ & $X$ & $g(Y)$ & $b_1(S)$ & $b_2(S)$ 
  
        \tabularnewline \hline

        $- \infty$  & (Elliptic) ruled surface & $\proj^1$ & $0$ & $2$ & $2$ 

        \tabularnewline \hline

        $0$ & Quasi-hyperelliptic & Rational with a cusp & $0$ & $2$ & $2$ 
        
        \tabularnewline \hline  

        $0$ & Hyperelliptic & Elliptic curve & $0$ & $2$ & $2$ 

        \tabularnewline \hline
        
        $0$ & Abelian surface & Elliptic curve & $1$ & $4$ & $6$ 

        \tabularnewline \hline  

        $1$ & Properly elliptic surface & Any other $G$-normal & $\geq 0$  & $2+2g(Y)$ & $2+4g(Y)$ 

        \tabularnewline \hline

	\end{tabular}
\end{center}

\medskip

In the case of abelian surfaces, $X$ is an elliptic curve on which the finite subgroup scheme $G\subset E$ acts by translations. For properly elliptic surfaces, namely relatively minimal surfaces of Kodaira dimension one, $X$ is any $G$-normal curve with arithmetic genus $p_a(X)\geq 2$.
\end{theoremA}

Let us mention that the last row of the above table actually contains an infinite family of examples of elliptic surfaces which are pairwise not birational to each other: indeed, as justified in \Cref{rem:michel_last}, the genus $g(Y)$ can be arbitrarily large, and $p_a(X)\geq 2$ as well. However, it is difficult to say more as there is no classification of $G$-normal curves.

\medskip

To fully understand these elliptic surfaces, it is useful to determine some additional invariants. In the case of elliptic surfaces with Kodaira dimension zero, a list of possible configurations of invariants was recently obtained in \cite{Mitsui_kappa=0}. 

\medskip

The second main result of this paper deals with the case of quotients by a \emph{diagonalizable} group scheme $G \subset E$. Under this additional assumption, we can extract more information on the surface $S$ and its Picard scheme.

\begin{theoremA}\label{theoremC}
    Under the assumptions of \Cref{TheoremA} and assume moreover that $G$ is diagonalizable. Then the strongly isotrivial elliptic fibration $f\colon S\to Y \defeq X/G$ has only tame fibers and the following equalities hold:
    \begin{itemize}
        \item $\chi(\mathcal{O}_S) = 0$,
        \item $q(S)= 1 +g(Y)$,
    \end{itemize}
    and the Picard scheme $\underline{\Pic}_S$ is reduced. Moreover, for every $y\in Y$, the group $\Pic^0(f^{-1}(y))$ is isomorphic to $E$. 
    Finally, the elliptic surface $S$ has Picard rank
    \[
        \rho(S) = \rho(E\times X) = 2 + \rank \mathrm{Hom}_{gp}(\mathrm{Alb}(X),E).
    \]
\end{theoremA}

\smallskip

Our article is organized as follows. In \Cref{Section:prelimimaries}, we recall some classical definitions concerning group scheme actions, as well as fundamental properties of $G$-normal curves. In particular, in the case where $G$ is a finite diagonalizable group scheme, Brion provides a complete local description of $G$-normal curves in \cite{diagos}. 

\medskip

In \Cref{Section:ellipticsurfaces}, we first deal with the case of elliptic ruled surfaces, which has been studied in \cite{Maruyama,Katsuro_Ueno,TogashiUehara}. An elliptic surface ruled over an elliptic curve $E$ is either isomorphic to $\proj(\mathcal{O}_E\oplus \mathcal{L})$ for some line bundle $\mathcal{L}$ of finite order (notice that $\mathcal{L}$ may be trivial, in which case the ruled surface is the product $E\times \proj^1$), or to one among the \emph{Atiyah ruled surfaces} $\mathsf{A}_0$ and $\mathsf{A}_1$, which are exactly the two indecomposable $\proj^1$-bundles over $E$. 
The Atiyah surface $\mathsf{A}_1$ is obtained as the contracted product associated to an $E[2]$-action on $E \times \proj^1$. To the extent of our knowledge, this was known in characteristic $p\geq 3$, we extend it here to the case $p=2$. In order to do this, we give an explicit description of the $E[2]$-action on $\proj^1$ in the case of an ordinary elliptic curve, for which the two-torsion subgroup is isomorphic to $\Z/2\Z\times \boldsymbol{\mu}_2$. In the case of a supersingular elliptic curve, the subgroup $E[2]$ is infinitesimal, and it is a nontrivial extension of $\boldsymbol{\alpha}_2$ by itself. In order to get an explicit action on the projective line, we make use of results by \cite{Bianca,Bianca_Dajano}. 

\medskip

We obtain that every elliptic ruled surface can be written as a contracted product $F\times^G \proj^1$, where $F$ is an elliptic curve isogeneous to $E$ such that $F/G=E$ and $G\subset F$ is a finite subgroup scheme. This allows us to show that a smooth projective surface $S$ is isomorphic to a contracted product $F\times^G X$, where $X$ is a $G$-normal curve, if and only if $S$ is a relatively minimal surface equipped with a strongly isotrivial elliptic fibration.

\medskip

Using the explicit description of the surface as a contracted product, we determine the \emph{Betti numbers} of $S$, using étale cohomology with $l$-adic coefficients. First, we show that $S$ and $E\times Y$ have the same Betti numbers. Then combining the comparison theorem, a result of lifting to characteristic zero for curves, together with the proper base change theorem for étale cohomology, we reduce the computation of Betti numbers to singular cohomology. 

\medskip

Next, we recall some generalities on the dualizing sheaf, which is a generalization of the canonical sheaf and exists for $G$-normal curves. We show that the canonical ring of $S$, namely
\[
R(S,\omega_S) \defeq \bigoplus_{n=0}^\infty H^0(S,{\omega_S}^n),
\]
is a finitely generated $k$-algebra, which is isomorphic to the $G$-invariant part of the ring of sections $R(X,\omega_X)$. Denoting respectively by $\kappa$ and $\kappa'$ the Kodaira dimension and the Iitaka dimension of the dualizing sheaf, this implies the numerical equality
\[
\kappa(S) = \kappa'(X).
\]
This leads to \Cref{TheoremA}.

\medskip

A natural question is to describe the automorphism groups $\Aut(S)$. In the case of ruled surfaces, the groups $\Aut(S)$ have been studied by Maruyama; and the cases of (quasi)-hyperelliptic surfaces by Bennett-Miranda over $\C$ and by Martin as group schemes; see \cite{Maruyama,Bennett_Miranda,Martin}. 

The connected component of the identity, denoted by $\Aut^0(S)$, is naturally equipped with a structure of smooth algebraic group. Those automorphism groups play a central role in the classification of connected algebraic subgroups of $\Bir(S)$, particularly those that are maximal with respect to inclusion within $\Bir(S)$, the so-called \emph{maximal connected algebraic subgroups} of $\Bir(S)$. 

\medskip

The pairs $(S,\Aut^0(S))$, where $S$ is a relatively minimal surface and $\Aut^0(S)$ is a maximal connected algebraic subgroup of $\mathrm{Bir}(S)$, are classified in \cite{fong_1} under the assumption that $\kappa(S)<0$, or that $\kappa(S)\geq 0$ and the characteristic of the base field is zero. If $S$ is relatively minimal and $\kappa(S)\geq 0$, then $\Aut(S)=\Bir(S)$; thus $\Aut^0(S)$ is the unique maximal connected algebraic subgroup. Moreover, $\Aut^0(S)$ is an abelian variety of dimension at most $2$. 

\medskip

Using \Cref{TheoremA}, we determine the birational classes of $S$ according to the dimension of $\Aut^0(S)$ and obtain the generalization of the above classification for $\kappa(S)\geq 0$ in positive characteristic:

\begin{corollaryA}\label{CorollaryB}
    The pairs $(S,\Aut^0(S))$, where $S$ is a relatively minimal surface with $\kappa(S)\geq 0$, are classified as follows: 
    \medskip
    \begin{center}
    \begin{tabular}{| c | c | c | }
		
        \hline
  
		$\kappa(S)$ & $S$ & $\Aut^0(S)$
    
        \tabularnewline \hline

        $0$ & Enriques surface & Trivial

        \tabularnewline \hline

        $0$ & K3 surface & Trivial
        
        \tabularnewline \hline

        $0$ &  Quasi-hyperelliptic surface & Elliptic curve

        \tabularnewline \hline

        $0$ &  Hyperelliptic surface & Elliptic curve

        \tabularnewline \hline
        
        $0$ & Abelian surface & Abelian surface

        \tabularnewline \hline  

        $1$ & Properly elliptic surface & Elliptic curve or trivial

        \tabularnewline \hline
        
        $2$ & General type surface & Trivial
        
        \tabularnewline \hline
	\end{tabular}
\end{center}

\medskip

If $S$ is an abelian surface, then $\Aut^0(S) = S$ acts on itself by translations. Moreover, $\Aut^0(S)$ is an elliptic curve $E$ if and only if $S$ is a contracted product $E\times^G X$ where $E$ acts on the first factor, $G\subset E$ is a finite subgroup scheme and $X$ is a $G$-normal curve but not an elliptic curve on which $G$ acts by translations.
\end{corollaryA}

\medskip

Apart from elliptic ruled surfaces, for which we are able to consider explicit quotients by an arbitrary finite group scheme, we restrict ourselves to the case of quotients by a \emph{diagonalizable} finite subgroup scheme $G \subset E$. In that case, we can use representation theory tools and thus exploit the local description given by \cite{diagos}.

\medskip

In \Cref{Section:diagonalizable}, we assume that $G$ is diagonalizable. Using techniques developed in \cite{diagos}, in particular the formula \emph{à la Hurwitz} for the quotient $\pi\colon X\to Y$, we describe a relation between the Kodaira dimension of $S$, the genus of $Y$ and the number of multiple fibers, as follows.

\newpage

\begin{propositionA}\label{PropositionB}
    Under the assumptions of \Cref{TheoremA}, assume moreover that $G$ is diagonalizable and infinitesimal. Then $\kappa(S)=1$, i.e.\ $S$ is a properly elliptic surface, if one of the following conditions is satisfied: 
    \begin{enumerate}
        \item either $g(Y)\geq 2$,
        \item or $g(Y)=1$ and there is at least one multiple fiber,
        \item or $Y=\proj^1$ and there are at least five multiple fibers (four multiple fibers if $p\geq 3$, three multiple fibers if $p\geq 5$).
    \end{enumerate}
\end{propositionA}

Making use of the dualizing sheaf of the $G$-normal curve, as well as the Hurwitz formula for the quotient $ X \rightarrow Y$, we compute the \emph{irregularity} $q(S)$ and the \emph{Euler characteristic} $\chi(\mathcal{O}_S)$. 
Using the notions of $G$-linearized line bundles and of Albanese variety, we are able to study the Picard scheme of the surface $S$; in particular, we see that it is smooth and that the Picard rank only depends on the $G$-normal curve $X$ and on the elliptic curve $E$. These results are summarized in \Cref{theoremC} above.

\smallskip

\begin{Acknowledgements}
   We are grateful to Jérémy Blanc, Michel Brion, Bianca Gouthier, Gebhard Martin, Ludvig Modin, Matthieu Romagny and Dajano Tossici for useful ideas and interesting discussions related to this article. We are also thankful to the anonymous referees for their careful reading and their comments. The first author is supported by the ERC StG Saphidir and the CNRS through the grant PEPS JC/JC, and is thankful to the Institut Fourier for their hospitality. 
\end{Acknowledgements}

\begin{conventions*}
    We work in the setting of algebraic groups over an algebraically closed field $k$ of characteristic $p> 0$: by (algebraic) \emph{group} we mean a group scheme of finite type over $k$, which is not necessarily reduced. If $G$ is an algebraic group, we denote by $X^*(G)$ its character group, i.e.\ the group of homomorphisms $G\to \Gm$. By \emph{variety}, in particular \emph{curve} and \emph{surface} if respectively of dimension $1$ and $2$, we mean a separated integral scheme of finite type over $k$. We assume all of our varieties to be \emph{projective}, unless otherwise stated. 
\end{conventions*}

\section{Preliminaries}\label{Section:prelimimaries}

\subsection{Algebraic group actions} 

First, let us recall some classical definitions about algebraic group actions. For an algebraic group $G$, a \emph{$G$-scheme} is a scheme $X$ over $k$ equipped with a $G$-action
\[
a \colon G \times X \longrightarrow X, \quad (g,x)\longmapsto g\cdot x,
\]
where the morphism $a$ is also defined over $k$. If the group scheme $G$ is finite, then the action morphism $a$ is finite and locally free. The $G$-action is said to be \emph{faithful} if every nontrivial subgroup acts nontrivially on $X$. 

\medskip

The \emph{stabilizer} $\Stab_G$ of the action is the preimage of the diagonal of the graph morphism 
\[
G\times X \longrightarrow X \times X, \quad (g,x) \longmapsto (x,g \cdot x).
\]
For $Y$ a closed subscheme of a scheme $X$, we say that $Y$ is \emph{$G$-stable} if $G \cdot Y = Y$ i.e.\ if the restriction of $a$ to $G \times Y$ factors through $Y$. When $G$ is finite, taking a point $x \in X(k)$, we denote by $G\cdot x$ its orbit, which is isomorphic to $G/\Stab_G(x)$ as schemes. We have that $x$ is $G$-stable if and only if it belongs to $X^G(k)$, the set of rational fixed points. On the other hand, we say that the action is \emph{free} at $x_0 \in X(k)$ if 
\[
\Stab_G(x_0) =1,
\]
where $\Stab_G(x_0)$ denotes the projection on $G$ of the fiber of $\Stab_G$ above the point $(x_0,x_0)$. We denote by $X_{\text{fr}}$ the set of free points of $X$, which is an open subset of $X$ and moreover it is $G$-stable.

\medskip

We say that a morphism of $G$-schemes $f \colon X \rightarrow Y$ is $G$-\emph{equivariant} if 
\[
f(g\cdot x) = g\cdot f(x) \quad \text{on } G \times X.
\]
Let us recall the following result, which is fundamental in equivariant birational geometry; see \cite[$7.2$]{Brion17}.

\begin{theorem}[Blanchard's Lemma]
\label{blanchard}
Let $g \colon T \rightarrow W$ be a proper morphism of schemes such that $g_*\mathcal{O}_T = \mathcal{O}_W$. Assume 
that $T$ is equipped with an action of a connected algebraic group $H$. Then there exists a unique $H$-action on $W$ such that the morphism $g$ is $H$-equivariant.
\end{theorem}

The automorphism group $\Aut(X)$ of a projective variety has a canonical structure of a reduced (equivalently, smooth) group scheme, locally of finite type over $k$. With respect to this structure, we denote as $\Aut^0(X)$ its connected component of the identity, which is an algebraic group. Our focus will not be on the group scheme structure, thus we mainly consider it simply as an abstract group. 

\begin{corollary}
    Under the assumptions of Theorem \ref{blanchard}, the morphism $g\colon T\rightarrow W$ induces a homomorphism of algebraic groups
    \[
        g_*\colon \Aut^0(T) \longrightarrow \Aut^0(W),
    \]
    such that for every $f\in \Aut^0(T)$, the following diagram commutes:
    \begin{center}
        \begin{tikzcd}
            T \arrow[rr, "f"] \arrow[d, "g"]&& T\arrow[d, "g"] \\
            W  \arrow[rr, "g_*(f)"] && W. 
        \end{tikzcd}
    \end{center}
\end{corollary}

Now assume $G$ to be finite over $k$ and let $\vert G \vert$ be its order, which is defined as being the dimension of $\mathcal{O}(G)$ as a $k$-vector space. Then $G$ lies in a unique exact sequence
\begin{align}
\label{etale_connected}
1 \longrightarrow G^0 \longrightarrow G \longrightarrow \pi_0(G) \longrightarrow 1
\end{align}
where $G^0$ is infinitesimal and $\pi_0(G)$ is a finite and constant group. 

\medskip

\begin{lemma}\label{lem:categorical_quotient}
Let $X$ be a $G$-scheme of finite type such that every $G$-orbit is contained in an open affine subset. Then there exists a categorical quotient by $G$,
\[
\rho \colon X \longrightarrow Y\defeq X/G,
\]
where $Y$ is a scheme of finite type. The morphism $\rho$ is finite and surjective, having as set-theoretic fibers the orbits of $G$. If in addition the action is free, $\rho$ is a \emph{(left) $G$-torsor}. Conversely, if the categorical quotient exists and is finite, then $X$ is covered by open affine $G$-stable subsets. 
\end{lemma}

The above result is \cite[III.2.6.1]{DG} and guarantees the existence of quotients in our case: if $X$ is a $G$-curve, then the categorical quotient always exists.

\begin{lemma}
\label{lin_red}
    Let $G$ be a finite and linearly reductive group. Then the operations of taking quotients and of taking closed $G$-stable subschemes commute.
\end{lemma}

\begin{proof}
    Let $X$ be a scheme equipped with a $G$-action and having a quotient by $G$ that we denote by $\pi \colon X \rightarrow Y$. Let $Z$ be a $G$-stable closed subscheme of $X$. The map $j$ making the diagram below commute
    \begin{center}
        \begin{tikzcd}
            X \arrow[rr, "\pi"] && Y \\
            Z \arrow[u, hookrightarrow] \arrow[rr, %"\pi_{\vert Z}"
            ] && Z/G \arrow[u, "j"]
        \end{tikzcd}
    \end{center}
    is well defined; we have to prove that it is a closed immersion.\\
    Since the map $\pi$ is finite, it follows by Lemma \ref{lem:categorical_quotient} that we can cover $X$ by $G$-stable affine open subsets and hence assume $X= \Spec A$ to be affine. Let $I$ be the ideal defining the closed subscheme $Z$, then we have the following short exact sequence of $G$-modules
    \[
    0 \longrightarrow I \longrightarrow A \longrightarrow A/I = \mathcal{O}(Z) \longrightarrow 0.
    \]
    Then one can take invariants, which is exact since $G$ is linearly reductive; this yields
    \[
      0 \longrightarrow I^G \longrightarrow A^G=\mathcal{O}(Y) \stackrel{\psi}{\longrightarrow} \mathcal{O}(Z)^G = \mathcal{O}(Z/G) \longrightarrow 0,
    \]where the map $\psi$ is exactly given by the morphism $j$. Thus the latter is a closed immersion.
\end{proof}

%----------------------------------------------------------------------------------------------------------------

\subsection{Equivariantly normal curves}
We gather here some material on $G$-normal curves, which can be found in \cite{Brion2}; the notion of $G$-normality is introduced in this paper, to deal with the fact that the action of a finite group scheme on a variety does not lift to the normalization in general.
\begin{definition}\cite[Definition 4.1]{Brion2}
    A $G$-variety $X$ is said to be \emph{$G$-normal} if every finite birational morphism of $G$-varieties $f \colon Y \rightarrow X$ is an isomorphism.
\end{definition}

By \cite[Proposition 4.2]{Brion2}, every $G$-variety $X$ admits a \emph{$G$-normalization}, i.e., a $G$-normal variety $X^\prime$ together with a finite, birational, $G$-equivariant morphism $\varphi\colon X^\prime \rightarrow X$, such that for any finite birational morphism $f\colon Z\to X$ of $G$-varieties, there exists a unique $G$-morphism $\psi\colon Z\to X'$ making the following diagram commutative:
\begin{center}
    \begin{tikzcd}
        && X^\prime \arrow[d, "\varphi"]\\
        Z \arrow[urr, "\psi"]\arrow[rr, swap,"f"] && X.
    \end{tikzcd}
\end{center}
By definition, the $G$-normalization is unique up to unique $G$-equivariant isomorphism.

\begin{lemma}[{\cite[Lemma 4.7 and Corollary 4.8]{Brion2}}]
    \label{normalization}
    Let $X$ be a $G$-normal variety. Then its normalization $\tau \colon \widetilde{X} \rightarrow X$ is bijective and purely inseparable; moreover, the quotient $X/G$ is normal.
\end{lemma}

\begin{lemma}[{\cite[Corollary 4.14]{Brion2}}]
\label{lci}
    For a $G$-curve $X$, the following conditions are equivalent:
    \begin{itemize}
        \item $X$ is $G$-normal;
        \item the sheaf of ideals $\mathcal{I}_{G\cdot x}$ is invertible for any closed point $x \in X$;
        \item the sheaf of ideals $\mathcal{I}_Z$ is invertible for any closed $G$-stable subscheme $Z$.
    \end{itemize}
    Moreover, every $G$-normal curve is a locally of complete intersection.
\end{lemma}

\begin{definition}
\label{contracted_product}
    Let $T$ be a right $G$-torsor and $F$ a $G$-scheme such that every $G$-orbit is contained in an open affine subset. Then the \emph{contracted product} of $T$ and $F$ is the scheme
    \[
    T \times^G F \defeq (T \times F)/G, \quad \text{where} \quad g \cdot (t,x) = (tg^{-1}, g\cdot x),
    \]
    which is equipped with two projections:
    \[
    \begin{array}{ccc}
        T\times^G F & \longrightarrow & T/G, \\ 
        T\times^G F & \longrightarrow & F/G.
    \end{array}
    \]
    The first one is a fibration, i.e.\ a morphism locally trivial for the fppf topology, with fiber $F$.
\end{definition}

In particular, in this text we consider the case where $G$ is a subgroup of an algebraic group $G^\flat$. For a subgroup $K \subset G$, we denote $\vert K \vert$ the order of $K$; in other words, the dimension as a $k$-vector space of the corresponding Hopf algebra.

\begin{proposition}[{\cite[Proposition 4.17]{Brion2}}]
\label{regular}
    Let $G$ be a subgroup scheme of a smooth connected algebraic group $G^\flat$. Then the contracted product
    \[
    S = G^\flat \times^G X
    \]
    exists for any $G$-curve $X$. Moreover, the curve $X$ is $G$-normal if and only if $S$ is smooth.
\end{proposition}

%----------------------------------------------------------------------------------------

\subsection{Local description in the diagonalizable case}
\label{sec:local}

Let us assume the group $G$ to be diagonalizable in this section; the following construction is due to \cite[Section 5]{diagos}. Let us consider a $G$-normal curve $X$; then to every non-free point $x\in X(k)$ we associate the subgroup
\[
H(x) \defeq \Stab_G(x) \subset G,
\]
together with a weight $\nu(x)$ of $G$ such that its restriction to $H(x)$ generates the character group of $H$. Moreover, let $n(x)$ denote the order of $H(x)$. Then, let 
\[
U(x) \defeq X_{\text{fr}} \cup \{ z \in X \, \vert \, H(z) = H(x) \text{ and } \nu(z) = \nu(x) \}.
\]
In particular, $X$ is the union of all such open $G$-stable subsets (in higher dimensions, the equality only holds in codimension one). By \cite[Theorem 1]{diagos}, we have the following: the group $H(x)$ is cyclic; moreover, on the open $G$-stable subset $U= U(x)$, the quotient morphism factors through
\begin{center}
    \begin{tikzcd}
            U \arrow[rr, "\varphi"] && U/H \arrow[rr, "\psi"] && U/G
    \end{tikzcd}
\end{center}
where $\varphi$ is a cyclic cover of degree $\vert H \vert$, and $\psi$ is a $G/H$-torsor. A $G$-normal curve $X$ is said to be \emph{uniform} if it can be obtained as the Zariski closure of some $U(x)$ as above. Moreover, by \cite[Proposition 6.5]{diagos} we have the following: if the Picard group of $Y$ has no $\vert G\vert$-torsion, then there is a bijection between uniform $G$-curves over $Y$ and reduced effective divisors on $Y$ with class divisible by $\vert G\vert$.

\subsection{Picard groups}
For a variety $X$, we denote by $\Pic(X)$ the group of isomorphism classes of line bundles over $X$. Let us recall a few facts concerning the Picard scheme of a projective variety; see \cite[Sections 4 and 5]{Kleiman}. 

\begin{definition}
    Let $B$ be a locally Noetherian scheme: we consider the \emph{relative Picard functor} of a $B$-scheme $Z$, is defined as 
    \[
    \Pic_{Z/B} \colon T \longmapsto \Pic(Z\times T) /\Pic(T),
    \]
    where $T$ is a $B$-scheme. If the associated sheaf for the \textit{fppf} topology is representable by a scheme, we denote it as $\underline{\Pic}_{Z/B}$ and call it the (relative) Picard scheme of $Z$ over $B$. If $B= \Spec k$, we denote it simply as $\underline{\Pic}_Z$ and call it the \emph{Picard scheme} of $Z$.
\end{definition}

\begin{theorem}
\label{Pic_curve}
    If $Z$ is projective over $k$, then the Picard scheme exists and it is locally of finite type. If moreover $Z$ is smooth, then the neutral component $\underline{\Pic}_Z^0$ of the Picard scheme is projective. 
    
    If $Z=C$ is a projective curve, then $\underline{\Pic}_C$ is smooth. If moreover $C$ is of genus at least $1$, then there is a natural closed embedding
    \[
    \begin{array}{ccc}
    C & \longhookrightarrow & \underline{\Pic}^1_C \\
    x & \longmapsto & \mathcal{O}_C(x).
    \end{array}
    \]
\end{theorem}

The above statement can be found in \cite[Corollary 4.18.4, Theorem 5.4, Remark 5.26]{Kleiman}. We denote as $\Pic^0(X)$ the connected component of the identity of the Picard group. The Néron-Severi group is the quotient $\mathrm{NS}(X)\defeq \Pic(X)/\Pic^0(X)$ and is a finitely generated abelian group; its rank is called the \emph{Picard rank} of $X$ and is denoted by $\rho(X)$.

\begin{lemma}
\label{Liu_unipotent}
    Let $X$ be a $G$-normal curve with $G$ infinitesimal and let $\tau \colon \widetilde{X} \rightarrow X$ be its normalization. Then the map $\tau^* \colon \underline{\Pic}^0_X \rightarrow \underline{\Pic}^0_{\widetilde{X}}$
    is surjective, with kernel a unipotent group. 
\end{lemma}

\begin{proof}
    This is a special case of \cite[Lemma 7.5.18]{Liu}, applied to the $G$-normal curve $X$, for which $\tau$ is a bijective morphism by \Cref{normalization}. The (reduced and projective) curve with ordinary singularities associated to $X$ is isomorphic to the smooth curve $\widetilde{X}$. Thus, the kernel of $\tau^*$ is unipotent of dimension $p_a(X)-p_a(\widetilde{X})$.
\end{proof}

%----------------------------------------------------------------------------------------

\section{Isotrivial elliptic surfaces and their Betti numbers}\label{Section:ellipticsurfaces}

\subsection{Elliptic ruled surfaces}
\label{section:ruled}

\begin{definition}
\begin{enumerate}
    \item A \emph{geometrically ruled surface} is a surface $S$ equipped with a surjective morphism $\pi\colon S\to C$ to a smooth curve $C$ such that each fiber is isomorphic to $\proj^1$. 
    \item If moreover there exists a morphism $S\to D$ to a smooth curve $D$ such that the generic fiber is a smooth curve of genus one, we say that $S$ is an \emph{elliptic ruled surface}.
\end{enumerate}
\end{definition}

\medskip
For every geometrically ruled surface $S$, there exists a rank-$2$ vector bundle $\mathcal{E}$ such that $\proj(\mathcal{E}) = S$ (see e.g. \cite[V. Proposition 2.2]{Hartshorne}). Therefore, the fibration $\pi\colon S\to C$ is locally trivial and we say that $\pi$ is a $\proj^1$-bundle over $C$.

A $\proj^1$-bundle $\proj(\mathcal{E})$ admits two disjoint sections if and only if $\mathcal{E}$ is \emph{decomposable}, i.e.\ if and only if $\mathcal{E}$ is isomorphic to the sum of two line bundles. In that case, we also say that $\proj(\mathcal{E})$ is decomposable; else, it is \emph{indecomposable}. By \cite[Theorem 11]{Atiyah}, there exist exactly two indecomposable $\proj^1$-bundles up to isomorphism over an elliptic curve $E$, which we denote by $\mathsf{A}_0$ and $\mathsf{A}_1$. They are respectively obtained by the projectivization of the indecomposable rank-$2$ vector bundles $\mathcal{E}_{2,0}$ and $\mathcal{E}_{2,1}$, which are of degree zero and one, and which fit into the short exact sequences
\[
\begin{array}{ccccccccc}
    0 & \to &\mathcal{O}_E & \to & \mathcal{E}_{2,0} & \to & \mathcal{O}_E & \to & 0, \\
     0 & \to &\mathcal{O}_E & \to & \mathcal{E}_{2,1} & \to & \mathcal{O}_E(z) & \to & 0;
\end{array}
\]
where $z\in E$, and the isomorphism class of $\mathsf{A}_1$ is independent of the choice of $z$.

\begin{remark}\label{rem:elliptic_ruled_surface}
Elliptic ruled surfaces are classified in \cite[Theorem 4]{Maruyama}; see also \cite[Propositions 2.10 and 2.15]{TogashiUehara}. A ruled surface $\pi\colon S\to E$ over an elliptic curve is an elliptic ruled surface if and only if $S$ is isomorphic to one of the following:
\begin{enumerate}
    \item \label{case.i} $\proj(\mathcal{O}_E\oplus \mathcal{L})$ for some $\mathcal{L}\in \Pic(E)$ of finite order,
    \item \label{case.ii} $\mathsf{A}_0$,
    \item \label{case.iii} $\mathsf{A}_1$.
\end{enumerate}

Over an algebraically closed field of characteristic $0$, the ruled surface $\mathsf{A}_0$ is not elliptic. The ruled surfaces $\proj(\mathcal{O}_E\oplus \mathcal{L})$ and $\mathsf{A}_0$ are elliptic if and only if the neutral components of their automorphism groups are anti-affine: see \cite[Examples 1.2.3 and 4.2.4]{BSU}. From the point of view of birational geometry, this can be understood as follows. The cone of curves of such a surface is two-dimensional: one extremal ray corresponds to the structural morphism of the $\proj^1$-bundle, while the other one corresponds to the contraction of the numerical class of a minimal section. However, every curve lying in one of these surfaces intersects a minimal section of self-intersection zero; see e.g.\ the proof of \cite[Lemma 2.14]{fong}. This excludes the existence of a contraction from these surfaces to a curve, other than the two structural rulings.

\medskip

By the results of \cite[Theorem 1.1]{TogashiUehara} (see also Remark 4.4 of \emph{loc.cit.}), the surfaces considered in \cite[Examples 4.7, 4.8, 4.9]{Katsuro_Ueno} are isomorphic to the elliptic ruled surfaces of \Cref{rem:elliptic_ruled_surface} (\ref{case.i}) and (\ref{case.ii}). 

We recall the constructions of Katsuro and Ueno. In case (\ref{case.i}), when $\mathcal{L}$ is of order $p$, the surface is isomorphic to $F\times^{\boldsymbol{\mu}_p} \proj^1$, where $F$ is an elliptic curve such that $F/\boldsymbol{\mu}_p = E$. In the next lemma, we recover this construction for any line bundle $\mathcal{L}$ of finite order. In case (\ref{case.ii}), the surface is isomorphic to $F\times^{G} \proj^1$, with $E= F/G$ and $G =\Z/p\Z$ or $\boldsymbol{\alpha}_p$, depending whether the base elliptic curve is ordinary or supersingular.
\end{remark}

\begin{lemma}\label{lem:dec_are_contracted}
    Let $E$ be an elliptic curve and $\mathcal{L} \in \Pic(E)$ of order $n\geq 2$. There exists an elliptic curve $F$ and an action of $\boldsymbol{\mu}_n$ on $F\times \proj^1$ such that \[
    \proj(\mathcal{O}_E \oplus \mathcal{L}) = F\times^{\boldsymbol{\mu}_n} \proj^1 \quad \text{and} \quad E=F/\boldsymbol{\mu}_n.
    \]
\end{lemma}

\begin{proof}
    Since $\Pic^0(E)=E$ admits an element of order $n\geq 2$, either $E$ is ordinary or $p$ does not divide $n$. This implies that $E[n]=(\Z/n\Z)\times \boldsymbol{\mu}_n$, where the constant cyclic group of order $n$ is generated by $\mathcal{L}$.  
    We set $F = E/(\Z/n\Z)$, which is also an elliptic curve; by construction the following is exact:
    \[
    0 \longrightarrow \Z/n\Z = \langle \mathcal{L} \rangle \longhookrightarrow E= \Pic^0(E) \longrightarrow F = \Pic^0(F) \longrightarrow 0.
    \]
    By duality, this yields a short exact sequence
    \[
    0 \longrightarrow \boldsymbol{\mu}_n \longrightarrow F \overset{\iota}{\longrightarrow} E \longrightarrow 0.
    \] 
    
    Next, let us consider the following commutative diagram:
    \begin{center}
    \begin{tikzcd}
    1 \arrow[r, rightarrow] & \Gm \arrow[r,rightarrow] & \GL_2 \arrow[r,rightarrow] & \PGL_2 \arrow[r,rightarrow] & 1 
    \\
    1 \arrow[r, rightarrow] & \Gm \arrow[r,rightarrow] \arrow[u, rightarrow] & \widetilde{\boldsymbol{\mu}_n} \arrow[r,rightarrow] \arrow[u, rightarrow] & \boldsymbol{\mu}_n \arrow[r,rightarrow] \arrow[u, rightarrow] & 1,
    \end{tikzcd}
    \end{center}
    where $\widetilde{\boldsymbol{\mu}_n}$ denotes the preimage of $\boldsymbol{\mu}_n$ in $\GL_2$ and acts diagonally on $F\times \proj^1$ via $\boldsymbol{\mu}_n$. 

    Write $n=mp^r$ with $m$ and $p$ coprime. Then $\boldsymbol{\mu}_n = \boldsymbol{\mu}_m \times \boldsymbol{\mu}_{p^r}$ and the restriction of $\widetilde{\boldsymbol{\mu}_n}$ above $\boldsymbol{\mu}_m$ is a central extension $H$ of $\boldsymbol{\mu}_m$ by $\mathbf{G}_m$. By \cite[III, \S 6, Corollaire 4.4]{DG}, such extensions are classified by 
    \[ 
    \mathrm{Ext}^1(\boldsymbol{\mu}_m,\mathbf{G}_m) = \mathrm{Ext}^1(\Z/m\Z,\mathbf{G}_m) = k^*/{k^*}^m =0, 
    \] 
    since $k$ is algebraically closed and $m$ is coprime to $p$. Hence $H =  \boldsymbol{\mu}_m \times \mathbf{G}_m$ and $\widetilde{\boldsymbol{\mu}_n}$ is an extension of the infinitesimal group $\boldsymbol{\mu}_{p^r}$ by $H$. By \cite[IV, §1, Proposition 4.5]{DG}, such extensions are classified by \[
    \mathrm{Ext}^1(\boldsymbol{\mu}_{p^r},H)=0;
    \]
    thus $\widetilde{\boldsymbol{\mu}_n} = \boldsymbol{\mu}_n \times \mathbf{G}_m$. In particular, the group $\widetilde{\boldsymbol{\mu}_n}$ is diagonalizable as well.

    Without loss of generality, the embedding of $\widetilde{\boldsymbol{\mu}_n}$ into $\GL_2$ can be see as a faithful action of $\widetilde{\boldsymbol{\mu}_n}$ on the two-dimensional vector space $k_0 \oplus k_\alpha$. With this notation, we mean that the action is trivial on the first copy of $k$, while the second factor is acted on via the character $\alpha\colon \widetilde{\boldsymbol{\mu}_n}\to \mathbf{G}_m$. The line bundle $\mathcal{L}$ is then isomorphic to $F\times^{\widetilde{\boldsymbol{\mu}_n}} k_\alpha$, where the structural morphism is identified with the projection on $F/\widetilde{\boldsymbol{\mu}_n}=E$. This implies that \[
    \proj(\mathcal{O}_F \oplus \mathcal{L}) = F\times^{\widetilde{\boldsymbol{\mu}_n}} \proj^1(k_0\oplus k_\alpha)=F\times^{\boldsymbol{\mu}_n} \proj^1(k_0\oplus k_\alpha)\]
    and we are done.
\end{proof}

\begin{comment}
\begin{lemma}
    Let $E$ be an elliptic curve, $\mathcal{L}\in \Pic^0(E)\setminus \{\mathcal{O}_E\}$ of order $n\geq 1$ and $S=\proj(\mathcal{O}_E \oplus \mathcal{L})$. Two cases arise:
    \begin{enumerate}
        \item If $p$ does not divide $n$, then $S$ is isomorphic to $F\times^{\Z/n\Z} \proj^1$, where $F= E/(\Z/n\Z)$;

        \item If $p$ divides $m$, then $E$ is ordinary and $S$ is isomorphic to $F\times^{\boldsymbol{\mu}_n} \proj^1$, where $F= E/(\Z/n\Z)$. 
    \end{enumerate}
\end{lemma}

\begin{proof}
    (1): Since $p$ does not divide $n$, the group of $n$-torsion of $E$ is isomorphic to $(\Z/n\Z)^2$. Then the subgroup $\Z/n\Z$ acts on $F$ by translations and $F/(\Z/n\Z) = E$. Moreover, the subgroup $\Z/n\Z\subset \Gm$ also acts naturally on $\mathbf{A}^1 \setminus \{0\}$ without fixed points and this action extends to $\proj^1$ by fixing $0$ and $\infty$. Then we obtain the following cartesian square
    \begin{center}
    \begin{tikzcd}
    F\times \proj^1 \arrow[r, rightarrow,"q"]\arrow[d, rightarrow, "p_1"] & F\times^{(\Z/n\Z)} \proj^1\arrow[d, rightarrow, "\pi"] \\
    F \arrow[r, rightarrow,"\iota"]& F/(\Z/n\Z) = E.
    \end{tikzcd}
    \end{center}    
\end{proof}
\end{comment}

Going back to the Atiyah surface $\mathsf{A}_1$, by \Cref{blanchard}, the structural morphism $\pi\colon \mathsf{A}_1 \to E$ induces a morphism of connected algebraic groups 
\[
\pi_*\colon \Aut^0(\mathsf{A}_1) \to \Aut^0(E) = E,
\]
which is surjective and has a finite kernel by \cite[Theorems 2 and 3]{Maruyama}. This implies that $\Aut^0(\mathsf{A}_1)$ is isomorphic to an elliptic curve $F$; so by \cite[Proposition 5.6]{Brion2}, there exist a finite subgroup scheme $G\subset F$ and a $G$-normal curve $Y$ such that $\mathsf{A}_1$ is isomorphic to the contracted product $F\times^G Y$.  

If $p\neq 2$, then the two torsion subgroup $E[2]\subset E$ is isomorphic to $(\Z/2\Z)^2$ and $\mathsf{A}_1$ is isomorphic to the contracted product \[
E\times^{E[2]} \proj^1,\]
where $E[2]$ acts on $\proj^1$ via $z \mapsto \pm z^{\pm 1}$: see e.g.\ \cite[\S 3.3.3]{fong_1}. 

To the extent of our knowledge, there is no description in the literature of $\mathsf{A}_1$ as a contracted product when $p= 2$. In that case, the two torsion subgroup scheme $E[2]$ is not constant. If $E$ is an ordinary elliptic curve, then $E[2] = \Z/2\Z \times \boldsymbol{\mu}_2$. 
Else $E$ is a supersingular elliptic curve and $E[2]$ is an infinitesimal group scheme which is a non-trivial extension of $\boldsymbol{\alpha}_2$ by itself (an explicit description is given in \cite[Lemma 4.2]{Bianca_Dajano}). 

\begin{lemma}\label{lem:A_1_ordinary}
    Let $E$ be an elliptic curve. Assume $p\neq 2$, or $p=2$ and $E$ is ordinary. Then $E[2]$ acts faithfully on $\proj^1$ via the embedding
    \[
    \begin{array}{ccc}
         \Z/2\Z \times \boldsymbol{\mu}_2 & \longrightarrow & \PGL_2 \\
         t \in \boldsymbol{\mu}_2 & \longmapsto & 
         \begin{pmatrix}
            t & 0 \\ 0 & 1
         \end{pmatrix} \\
         1 \in \Z/2\Z & \longmapsto & 
         \begin{pmatrix}
            0 & 1 \\ 1 & 0
         \end{pmatrix}.
    \end{array}
    \]
    This induces a diagonal action of $E[2]$ on $E\times \proj^1$ and the contracted product $E\times^{E[2]} \proj^1$ is isomorphic to the ruled surface $\mathsf{A}_1$.
\end{lemma}

\begin{proof}
    In $\PGL_2$, the following equalities hold:
    \[
    \begin{pmatrix}
        t & 0 \\ 0 & 1
    \end{pmatrix} 
    \begin{pmatrix}
        0 & 1 \\ 1 & 0
    \end{pmatrix}     
    =
    \begin{pmatrix}
        0 & t\\ 1 & 0
    \end{pmatrix}
    =
    \begin{pmatrix}
        0 & 1\\t&0
    \end{pmatrix}
    =
    \begin{pmatrix}
        0 & 1 \\ 1 & 0
    \end{pmatrix}
    \begin{pmatrix}
        t & 0 \\ 0 & 1
    \end{pmatrix},
    \]
    as $t \in \boldsymbol{\mu}_2$. This yields the embedding of $E[2]$ in $\PGL_2$ given in the statement, and hence an action of $E[2]$ on $\proj^1$. We obtain the following cartesian square:
    \begin{center}
    \begin{tikzcd}
    E\times \proj^1 \arrow[r, rightarrow,"q"]\arrow[d, rightarrow, "p_1"] & E\times^{E[2]} \proj^1\arrow[d, rightarrow, "\Tilde{p}"] \\
    E \arrow[r, rightarrow,"\pi"]& E/E[2],
    \end{tikzcd}
    \end{center}
    where $p_1$ denotes the projection on the first factor, $\tilde{p}$ the projection onto the first factor modulo $E[2]$, while $q$ and $\pi$ are the quotient maps by $E[2]$. 

    Now we claim that $E\times^{E[2]} \proj^1$ is isomorphic to $\mathsf{A}_1$. First, assume by contradiction that $\tilde{p}$ admits two disjoint sections $\sigma_1$ and $\sigma_2$. Then their pullbacks $\pi^*\sigma_1$ and $\pi^*\sigma_2$ are disjoint sections of the trivial $\proj^1$-bundle $p_1$; moreover, they are given by precomposition by $\pi$; hence, they are $E[2]$-invariant constant sections of the trivial bundle $p_1$. As the group $E[2]$ acts on $\proj^1$ without fixed point, the constant sections of $p_1$ are not obtained by pulling back sections of $\tilde{p}$. Hence we get a contradiction, which shows that the $\proj^1$-bundle $\tilde{p}$ is indecomposable. 

    Moreover, the $\proj^1$-bundle $\tilde{p}$ has no section of self-intersection zero (this may be seen using the formula in \cite[Proposition 1.10]{Debarre}). In particular, the $\proj^1$-bundle $\tilde{p}$ is not isomorphic to $\mathsf{A}_0\to E$, as the latter admits a unique minimal section of self-intersection zero: see e.g.\ \cite[Proposition 2.21]{fong_1}.
\end{proof}

\begin{remark}
\label{A0}
    Replacing the two-torsion subgroup scheme by its reduced structure in the proof above does not give the ruled surface $\mathsf{A}_1$. Instead, this construction is similar to \cite[Example 4.7]{Katsuro_Ueno} and we obtain the ruled surface $\mathsf{A}_0$ over an ordinary elliptic curve $E$. Indeed, set $F\defeq E/\boldsymbol{\mu}_2$. 
    Then the reduced component of the two torsion subgroup $F[2]_{\text{red}} =\Z/2\Z$ acts on $\proj^1$ via $z\mapsto z^{\pm 1}$. Also, by duality, $E=F/F[2]_{\text{red}}$.
    This yields the following cartesian square
    \begin{center}
    \begin{tikzcd}
    F\times \proj^1 \arrow[r, rightarrow,"q"]\arrow[d, rightarrow, "p_F"] & F\times^{F[2]_{\text{red}}} \proj^1\arrow[d, rightarrow, "p_{E}"] \\
    F \arrow[r, rightarrow,"h"]& E\defeq F/F[2]_{\text{red}}.
    \end{tikzcd}
    \end{center}
    Since the trivial $\proj^1$-bundle $p_F$ admits a unique $(\Z/2\Z)$-invariant constant section, we obtain by the projection formula that $p_E$ has a unique minimal section of self-intersection zero; see e.g.\ \cite[Proposition 1.10]{Debarre}. Therefore, $p_E$ is an indecomposable $\proj^1$-bundle isomorphic to $\mathsf{A}_0$ or $\mathsf{A}_1$, but the minimal sections of the ruled surfaces $\mathsf{A}_0$ and $\mathsf{A}_1$ have respectively self-intersection zero and one (see e.g.\ \cite[Propositions 2.18 and 2.21]{fong_1}); thus, 
    \[
    F\times^{F[2]_{\text{red}}} \proj^1=\mathsf{A}_0.\] 
\end{remark}

In the following computation, the parameter $t \in E[2]$ is intended in a functorial sense: for a $k$-algebra $R$, we consider $t \in E[2](R)$.

\begin{lemma}\label{lem:A_1_supersingular}
    Assume $p=2$ and let $E$ be a supersingular elliptic curve. Then the subgroup scheme $E[2]$ %embeds  as a scheme in $\AA^1$ and it 
    is a nontrivial extension of $\boldsymbol{\alpha}_2$ by itself. Moreover, $E[2]$ acts faithfully on $\proj^1$ via the embedding
    \[
    \begin{array}{ccc}
    E[2] & \longrightarrow & \PGL_2\\
    t & \longmapsto & 
    \begin{pmatrix}
        1 & t^2 \\ t & 1+t^3
    \end{pmatrix}.
    \end{array}
    \]
    This induces a diagonal action of $E[2]$ on $E\times \proj^1$, and the contracted product $E\times^{E[2]} \proj^1$ is isomorphic to the ruled surface $\mathsf{A}_1$. 
\end{lemma}

\begin{remark}
	The morphism in the statement above, since we are dealing with infinitesimal group schemes, needs an additional argument to be well-defined: if $R$ is a finitely generated $k$-algebra, then $\PGL_2(R)$ is not isomorphic to the quotient $\GL_2(R)/R^{\times}$; however, the natural short exact sequence describing $\PGL_2$ as a quotient of $\GL_2$ by $\mathbf{G}_m$ implies there is an inclusion of $\GL_2(R)/R^{\times}$ into $\PGL_2(R)$.
	
	The ideas of the proof below come from \cite[Example 5.11]{Bianca} and \cite{Bianca_Dajano}. We write here a less technical description which avoids the formalism of Hopf algebras, suggested as well by Gouthier and Tossici. \Cref{lem:A_1_supersingular} is due to them and is also partially contained in the PhD thesis of Gouthier.
\end{remark}

\begin{proof}[Proof of \Cref{lem:A_1_supersingular}]
    Since the $2$-torsion group $E[2]$ is a subscheme of $E$ of length four with reduced subscheme a point, so it is isomorphic to 
    \[
    \Spec \left(k[t]/t^4\right)
    \] as a scheme. By \cite[III, $\S$6, Corollary 7]{DG} applied to the case of an algebraically closed base field $k$, there are exactly four isomorphism classes of commutative unipotent infinitesimal groups obtained as extensions of $\boldsymbol{\alpha}_2$ by itself, corresponding to the cases $(0,0)$, $(1,0)$, $(0,1)$ and $(1,1)$ in the statement of the Corollary. In particular, these correspond to $\boldsymbol{\alpha}_2\times \boldsymbol{\alpha}_2$, to $\boldsymbol{\alpha}_4$, to 
    \begin{align}
    \label{FrobeniusW2}
    \ker (F \colon W_2 \longrightarrow W_2)
    \end{align}
    and finally to a fourth group which we do not need in this context. Concerning notation, $W_2$ denotes the ring scheme of Witt vectors of length two over $k$, while $F$ is the Frobenius homomorphism. Next, let us consider
    \[
    G=  \ker (F-V \colon W_2\longrightarrow W_2) = \Spec (k[T_0,T_1]/(T_0^2,T_1^2-T_0)),
    \]
    where $V$ is the Verschiebung homomorphism. In particular, the group law is given by the one on $W_2$, namely
    \[
    (t_0,t_1) \boxplus (s_0,s_1) = (t_0+s_0,t_1+s_1+t_0s_0).
    \]
    Both the group $G$ and $E[2]$ are nontrivial extensions of $\boldsymbol{\alpha}_2$ by itself. Moreover, they cannot be isomorphic to $\boldsymbol{\alpha}_4$ nor to the Frobenius kernel (\ref{FrobeniusW2}), since the Verschiebung morphism of these latter groups is trivial, while the one of $G$ and of $E[2]$ are not. Hence, both are isomorphic to the fourth group, and in particular $E[2] = G$.

    Since in $E[2]$ we have the equality $t_0 = t_1^2$, we can just work with the second coordinate and simply call it $t = t_1$. Thus, the group scheme structure is as follows :
    \[
    E[2] \simeq \Spec(k[t]/t^4), \quad t  \boxplus s = t+s+t^2s^2.
    \]
    We can verify that the embedding into $\PGL_2$ given in the statement respects this group law.
    On one hand, we have
    \begin{align}
        \label{computation_E2}
    t \boxplus s \longmapsto
    \begin{pmatrix}
        1 & t^2+s^2\\
        t+s+t^2s^2 & 1+t^3+st^2+s^2t+s^3
    \end{pmatrix}.
    \end{align}
    On the other hand, 
    \[
    \begin{pmatrix}
        1 & t^2\\ t & 1+t^3
    \end{pmatrix}
    \begin{pmatrix}
        1 & s^2\\ s & 1+s^3
    \end{pmatrix}
    = \begin{pmatrix}
        1+t^2s & s^2+t^2+t^2s^3\\
        t+s+st^3 & ts^2+1+t^3+s^3+t^3s^3
    \end{pmatrix}.
    \]
    Since we are in $\PGL_2$, we can multiply the last matrix by the nonzero scalar $1+t^2s$. Using the fact that $t^4=s^4=0$, we get exactly the same thing as in (\ref{computation_E2}).
    
    Therefore, the subgroup scheme $E[2]\subset E$ acts diagonally on the product $E\times \proj^1$, and as in the proof of Lemma \ref{lem:A_1_ordinary}, we obtain the following commutative square:
    \begin{center}
    \begin{tikzcd}
    E\times \proj^1 \arrow[r, rightarrow,"q"]\arrow[d, rightarrow, "p_1"] & E\times^{E[2]} \proj^1\arrow[d, rightarrow, "\Tilde{p}"] \\
    E \arrow[r, rightarrow,"\pi"]& E/E[2].
    \end{tikzcd}
    \end{center}
    Since the action of $E[2]$ on $\proj^1$ has no fixed point, it follows that the constant sections of $p_1$ are not obtained by pulling back sections of $\tilde{p}$. Thus there is no section of $\tilde{p}$ of self-intersection zero, and this implies once more that $E\times^{E[2]} \proj^1 = \mathsf{A}_1$.
\end{proof}

\begin{remark}
	The constructions of Lemmas \ref{lem:A_1_ordinary} and \ref{lem:A_1_supersingular} give an explicit embedding of $E[2]$ into $\PGL_2$. An alternative argument for the existence of such embedding is the following: $E[2]$ commutes with the sign involution $\sigma$ on $E$, so $E[2]$ acts faithfully on the quotient $E/\sigma = \proj^1$.
\end{remark}

The following proposition seems to be a folklore result, known by the experts of surfaces in positive characteristic, but we cannot locate any reference.

\begin{proposition}\label{prop:ellipticruledsurface}
    Let $S$ be a ruled surface. Then $S$ is elliptic if and only if there exist an elliptic curve $F$ and a finite subgroup scheme $G\subset F$ such that $S=F\times^G \proj^1$. In particular, $S$ is equipped with a faithful action of the elliptic curve $F$.
\end{proposition}

\begin{proof}
    By Remark \ref{rem:elliptic_ruled_surface}, $S$ is elliptic if and only if $S$ is isomorphic to $\proj(\mathcal{O}_E\oplus \mathcal{L})$ with $\mathcal{L}\in \Pic^0(E)\setminus \{\mathcal{O}_E\}$ of finite order, or $\mathsf{A}_0$ or $\mathsf{A}_1$. Each of these surfaces can be expressed in the form of a contracted product as in the statement, by \cite[Examples 4.7 and 4.9]{Katsuro_Ueno} (see also Remark \ref{rem:elliptic_ruled_surface}) for the case $S=\mathsf{A}_0$, and Lemmas \ref{lem:dec_are_contracted}, \ref{lem:A_1_ordinary} and \ref{lem:A_1_supersingular} for the remaining cases.
\end{proof}

\subsection{Strongly isotrivial elliptic fibrations} 
Let $S$ be a smooth projective surface and $Y$ be a smooth projective curve with generic point $\eta$. 

\begin{definition}\label{def:stronglyiso}
	A morphism
	\[
	f \colon S \longrightarrow Y
	\]
	is a \emph{strongly isotrivial elliptic fibration} if there exists an elliptic curve $E$ over $k$ such that 
	\[
	(\underline{\Pic}^0_{S/Y})_\eta = E \times_{\Spec k} \Spec k(Y).
	\]
\end{definition}

In particular, such fibration is automatically an \emph{elliptic fibration}, i.e.\ the generic fiber is a torsor under an elliptic curve.

We would like to work with a regular action of an elliptic curve on this family of surfaces; it turns out that it is enough to assume minimality in order to have one.

\begin{lemma}
    \label{lem:first}
    Let $S$ as above be relatively minimal. Then the action of $E$ on the generic fiber of $S$ induces a regular faithful action of $E$ on $S$.% Then the above elliptic curve $E$ acts faithfully
\end{lemma}

\begin{proof}
   If $\kappa(S)<0$, the statement follows from Proposition \ref{prop:ellipticruledsurface}. From now on, assume that $\kappa(S) \geq 0$. Then $\Pic^0 _{S_\eta/\eta}$
   is an elliptic curve over $\eta$ and it comes equipped with a natural action on 
   \[
   S_\eta \stackrel{\sim}{\longrightarrow} \underline{\Pic}^1_{S_\eta/\eta}.
   \]
   By the hypothesis on $f$, there is an elliptic curve $E$ (over $k$) such that
   \[
   \underline{\Pic}^0 _{S_\eta/\eta} = E \times_{\Spec k} \Spec k(\eta).
   \]
   Thus we can reformulate the above by saying that the curve $E_\eta$ acts on $S_\eta$; by the associativity of the fiber product, this yields a morphism
   \begin{align}
   \label{action_genfiber}
   E \times_{\Spec k} S_\eta \longrightarrow S_\eta.
   \end{align}
   The last map is an action of $E$ of $S_\eta$: indeed, $E_\eta$ acts on $S_\eta$, and all the isomorphisms considered just above commute with taking products. 
   Next, let us show that this action extends to a non-empty open subset of $S$. Let $\Gamma$ be the graph of (\ref{action_genfiber}) inside of $E \times S \times S$; then we have the following commutative diagram, where $pr_Y$ is the composition of the projection on $S$ and the fibration $f$.
   \begin{center}
   \begin{tikzcd}
       & \Gamma \arrow[d, "\sigma"] \arrow[r, hookrightarrow] & E \times S \times S \arrow[dl, "pr_{E \times S}"]\\
       Y & E \times S \arrow[l, "pr_Y"] & 
   \end{tikzcd}
   \end{center}
   By construction, $\sigma$ is an isomorphism over $E \times S_\eta$; hence there exists an open subset
   \[
   E \times S \supset U \supset E \times S_\eta
   \]
   such that $\sigma$ defines an isomorphism over $U$ (namely, the open subset $U$ on which the birational inverse of $\sigma$ is defined). Let $Z \defeq (E \times S ) \setminus U$, then as underlying sets:
   \[
   \vert Z \vert \subset \vert E \times (S \setminus S_\eta) \vert.
   \]
   Therefore, $pr_Y(Z) \subset Y$ does not contain the generic point $\eta$; hence, it must coincide with a finite subset $W$ of $Y$. The open subset $U$ then contains $E \times f^{-1}(Y \setminus W)$; hence there is a well-defined morphism 
   \[
   \alpha \colon E \times f^{-1} (Y \setminus W) \longrightarrow S
   \]
   i.e.\ a rational $E$-action on $S$. Composing $\alpha$ with $f$ yields an $E$-invariant morphism; hence, there is a unique factorisation through $Y \setminus W$. Thus, $E$ acts on $f^{-1}(Y\setminus W)$. 

    By the Weil regularization theorem, there exists a birational model $T$ of $S$ on which $E$ acts regularly. By \cite[Corollary 3]{Brion_actions_on_normal_varieties}, we may assume that $T$ is normal and projective. Replacing by its desingularization, which is obtained by successive blowups of the singular points and normalizations (see \cite[Remark B p.155]{Lipman}), we may furthermore assume that $T$ is smooth. Contracting successively the $(-1)$-curves, we arrive to a relatively minimal surface equipped with a regular action of $E$. Since $\kappa(S)\geq 0$ and $S$ is relatively minimal, it follows that the obtained surface is isomorphic to $S$ itself; hence, the action of $E$ on the non-empty open subset $f^{-1}(Y \setminus W)$ extends to the whole surface $S$, and we are done. 
\end{proof}

Thanks to \Cref{lem:first}, from now on we can (and we do) place ourselves in the following setting: let $E$ be an elliptic curve, acting faithfully on a smooth projective surface $S$. As illustrated in \cite[Section 5]{Brion2}, there exist a finite subgroup scheme $G \subset E$ and a $G$-normal curve $X$ such that
\[
S = E \times^G X,
\]
where the isomorphism is $E$-equivariant.

\begin{remark}
\label{diagram}
Let us fix the following notation for the maps involving $S$.
\[
\begin{tikzcd}
    E \times X \arrow[dd, "pr_X"] \arrow[r,"q"] & S\arrow[dd,"f"] \arrow[r,"h"] & E/G\\
    &&\\
    X \arrow[r, "\pi"] & Y \defeq X/G &
\end{tikzcd}
\]
The quotient morphism $q$ is a $G$-torsor, hence it is a finite morphism. 
The map $h$ is a surjective morphism locally trivial for the $fppf$ topology, with fiber $X$, and whose target is an elliptic curve isogenous to $E$. The morphism $\pi$ is the quotient by the group $G$, it restricts to a $G$-torsor 
\[
X_{\text{fr}} \longrightarrow Y_{\text{fr}}
\]
over the largest open $G$-stable subset $X_{\text{fr}}$ on which $G$ acts freely, and $Y_{\text{fr}}$ is open in $Y$. As noticed in the proof of \cite[Proposition 5.6]{Brion2}, the $G$-action is generically free (which is not automatically implied by faithfulness of the action, since we are dealing with group schemes) so that $X_{\text{fr}}$ is non-empty.  
Finally, the map $f$ is the categorical quotient by $E$; it is an elliptic fibration with fiber $E$ over $Y_{\text{fr}}$, while on the complement it might have multiple fibers; we describe these fibers in \Cref{nowild} below.
\end{remark}

\begin{remark}
\label{rem:nosection}
    A crucial point to notice is that the elliptic fibration $f\colon S\to Y$ does not admit a section in general. Let us assume such a section $\sigma$ exists, then it defines a map $\xi \colon E \times Y \to S$, defined by the action of $E$ and by $\sigma$. Since the action of $E$ is generically free, i.e., it is free at the generic point $\eta$, the restriction $ E \times \eta \to S_\eta$ is an isomorphism. Thus, the morphism $\xi$ is a birational morphism and decomposes as a product of contractions of $(-1)$-curves. The surfaces $E\times Y$ and $S$ being minimal, this implies that $\xi$ is an isomorphism between $S$ and $E \times Y$, which is in general not the case.
\end{remark}

Let us now prove a kind of a converse implication to \Cref{lem:first}.

\begin{lemma}
\label{lem:minimal}
    Assumption as above. Then any curve on $S$ has non-negative self-intersection; in particular, $S$ is minimal.
\end{lemma}

\begin{proof}
    Let us assume that the surface $S$ is not minimal and let us consider a $(-1)$-curve $C$ and the morphism
    \[
    g \colon S \longrightarrow S^\prime
    \]
    obtained by contracting $C$ to a point $s$. By \Cref{blanchard} applied to the $E$-action, the curve $C$ must be $E$-stable, hence 
    \[
    \Stab_E(s) = E.
    \]
    However, the stabilizer of a point (for a faithful action of an algebraic group) must be a linear subgroup, hence we get a contradiction.
\end{proof}

Let us also mention that such surfaces $S$ have few rational curves. Indeed, either $X$ is rational, in which case the rational curves are exactly the fibers of $h$, or $X$ is not rational, which implies that $S$ does not contain any rational curve.

\newpage

\begin{proposition}
    Let $S$ be a smooth projective surface. The following are equivalent:
    \begin{enumerate}
        \item There exists an elliptic curve $E$ acting faithfully on $S$,
        \item There exist an elliptic curve $E$, a finite subgroup scheme $G\subset E$ and a $G$-normal curve $X$ such that $S = E\times^G X$,
        \item The surface $S$ is relatively minimal and is equipped with a strongly isotrivial elliptic fibration $S\to Y$.
    \end{enumerate}
\end{proposition}

\begin{proof}
    (1) implies (2) is proven in \cite[Proposition 5.6]{Brion2}; and conversely, the elliptic curve $E$ acts on the first coordinate of the contracted product of $E\times^G X$. The equivalence with (3) follows from Lemma \ref{lem:first}.
\end{proof}

\subsection{Betti numbers}
We now compute the Betti numbers of the elliptic surface $S$.
Let $l$ be a prime number distinct from $p$. As in \cite[Section 3.2]{Liedtke}, we define the \emph{$i$-th Betti number} of $S$ as 
\[
b_i(S) \defeq \dim H^i_{\text{ét}} (S,\Q_l), 
\]
which is independent of the choice of $l$. It is useful to fix the following notation for a scheme $X$:
\[
H^r_{\text{ét}} (X,\Q_l) = (\lim_{\longleftarrow} H^r_{\text{ét}} (X, \Z/l^n\Z)) \otimes_{\Z_l} \Q_l = H^r_{\text{ét}} (X,\Z_l) \otimes \Q_l.
\]
Let us recall a few results of étale cohomology which we are going to need in our computation below. The first one is a rather fundamental property, which turns out useful when dealing with infinitesimal group schemes. It is the \emph{topological invariance} of étale cohomology \cite[\href{https://stacks.math.columbia.edu/tag/03SI}{Proposition 03SI}]{stacks-project}: let $j \colon Z \rightarrow X$ be a universal homeomorphism (for example, the closed immersion defined by a nilpotent sheaf of ideals) and let $\mathcal{F}$ be an abelian sheaf on $X$. Then
\begin{align}
\label{topologicalinvariance}
H^r_{\text{ét}} (X, \mathcal{F}) = H^r_{\text{ét}} (Z, j^*\mathcal{F}) \quad \text{for all } r.
\end{align}
The second one is an application of the proper base change Theorem stated in \cite[\href{https://stacks.math.columbia.edu/tag/0DDF}{Lemma 0DDF}]{stacks-project}: let $f \colon X \rightarrow S$ be a proper morphism of schemes, and let $s$ be a geometric point of $S$. Then for any torsion abelian sheaf $\mathcal{F}$ on $X_{\text{ét}}$ we have
\begin{align}
\label{cohomology_fibers}
(R^i f_* \mathcal{F} )_s = H^i_{\text{ét}}(X_s, \mathcal{F}_s) \quad \text{for all } i.
\end{align}

Let us first clarify what we mean by a $G$-action on the étale cohomology with coefficients in $\Q_l$. If $G$ is a constant group, one can look at \cite[Chapter 5]{Tohoku}. By \emph{dévissage}, it is enough to define such an action for an infinitesimal group scheme.

\begin{lemma}
    Let $G$ an infinitesimal group scheme acting on a quasi-projective variety. Then there is a natural action of $G$ on $H^i_{\text{ét}}(T,\Q_l)$ for all $i \geq 0$.
\end{lemma}

\begin{proof}
    It suffices to define a natural $G$-action on the groups $H^i_{\text{ét}}(T,\mathcal{F})$, where $\mathcal{F} \defeq \Z/l^n\Z$, for all $n$ and all $i$; then the desired $G$-action is obtained by taking the inverse limit. Since $T$ is by assumption quasi-projective, by \cite[Chapter III, Theorem 2.17]{Mil}, the étale cohomology groups can be computed using Čech cohomology 
    $\check{H}^\bullet(T,\mathcal{F})$. By construction, the latter is defined as being the inverse limits of the Čech cohomologies $\check{H}^i(\mathcal{U},\mathcal{F})$, where $\mathcal{U} = (U_i \to T)_i$ ranges among all the étale covers of $T$. Thus, it suffices to define the $G$-action on a fixed $\check{H}^i(\mathcal{U},\mathcal{F})$. For each étale map $U \to T$, the $G$-action on $T$ lifts uniquely to a $G$-action on $U$, thanks to the infinitesimal lifting property of étale morphism; see \cite[Remark 18.4(3)]{GortzWedhorn2}. Next, let us consider the Čech complex $\mathcal{C}^\bullet(\mathcal{U},\mathcal{F}) = (\mathcal{C}^j(\mathcal{U},\mathcal{F}))_j$ which computes the cohomology for the étale cover $\mathcal{U}$. Each $\mathcal{C}^j(\mathcal{U},\mathcal{F})$ is now equipped with a canonical action of $G$ coming from the action on $T$. It follows that the same holds for $\check{H}^i(\mathcal{U},\mathcal{F})$, and we are done.
\end{proof}

\begin{lemma}
\label{etale_infinitesimal} 
    Let $G$ be an infinitesimal group scheme acting on a variety $T$ and let $\sigma \colon T \rightarrow S$ be the quotient by $G$. Then
    \[
    H^i_{\text{ét}} (S,\Q_l) = H^i_{\text{ét}} (T, \Q_l) \quad \text{for all } i.
    \]
\end{lemma}

\begin{proof}
    The quotient morphism $\sigma$ being a universal homomorphism, this follows from the topological invariance property (\ref{topologicalinvariance}).  % Let us start by showing the following equality for the constant sheaf $\Z/l^n\Z$:
    %\[
    %R^i\sigma_* \Z/l^n\Z = 0 \quad \text{for all } i >0;
    %\]
    %which guarantees that the Leray spectral sequence degenerates at $E_2$. Let $s$ be a geometric point of $S$. Then the fiber above the point $s$, thanks to the fact that $G$ is infinitesimal, has reduced subscheme equal to $\Spec \kappa (s)$. Thanks to (\ref{cohomology_fibers}), together with the topological invariance of étale cohomology (applied to the nilpotent ideal defining the reduced subscheme) we get
    %\[
    %(R^i\sigma_* \Z /l^n\Z)_s = H^i_{\text{ét}} (T_s,\Z/l^n\Z) = H^i_{\text{ét}} (\Spec \kappa(s),\Z/l^n\Z) = 0,  
    %\]
    %where the last equality comes from the fact that the higher étale cohomology of the spectrum of an algebraically closed field is trivial. Hence, by the Leray spectral sequence we get 
    %\[
    %H^r_{\text{ét}} (T, \Z/l^n \Z) = H^r_{\text{ét}} (S, \sigma_* \Z/l^n\Z) \quad \text{for all } r.
    %\]
    %By taking the inverse limit to $\Z_l$ on both sides, using \Cref{lin_red} and then taking the tensor product by $\Q_l$, we get
    %\[
    %H^r_{\text{ét}} (T, \Q_l) = H^r_{\text{ét}} (S, %\sigma_* \Q_l) \quad \text{for all } r.
    %\]
    %Finally, we take $G$-invariants on both sides, which gives the desired equality because $\sigma_* \Q_l = \Q_l$; again by the argument of topological invariance, because the reduced part of the fiber is $\Spec \kappa(s)$.
\end{proof}

\begin{lemma}
    \label{etale_finite}
 Let $G$ be a finite constant group acting on a variety $T$ and let $\sigma \colon T \rightarrow S$ be the quotient by $G$. Assume that $S$ exists in the category of schemes. Then
 \[
 H^r_{\text{ét}} (S,\Q_l) = H^r_{\text{ét}} (T, \Q_l)^G \quad \text{for all } r.
 \]
\end{lemma}

\begin{proof}
    %We proceed as in the proof of \Cref{etale_infinitesimal}:
    Let $s$ be a geometric point of $S$; then the fiber above $s$ satisfies
    \[
    (T_s)_{\text{red}} = (G/\Stab_G(s)) \times \Spec \kappa(s),
    \]    
    where $\Stab_G(s)$ denotes the stabilizer at the point $s$. The underlying topological space of $T_s$ consists of a finite number of copies of $\Spec \kappa(s)$, which are permuted by the $G$-action. This, together with (\ref{topologicalinvariance}), implies that
    \[
    H^i_{\text{ét}} (T_s, \Z/l^n\Z) = 0 \quad \text{for all } i >0.
    \]
    
    The above vanishing, together with the equality (\ref{cohomology_fibers}), guarantees that the Leray spectral sequence
    \[
    E_2^{p,q}= H^p_{\text{ét}} (S, R^q\sigma_* \Z/l^n\Z) \, \Longrightarrow \, H^{p+q}_{\text{ét}} (T,\Z/l^n\Z)
    \]
    degenerates at the page $E_2$. Next, taking the inverse limits to $\Z_l$, tensoring by $\Q_l$ and finally taking $G$-invariants, we obtain
    \[
    H^r_{\text{ét}} (T,\Q_l) = H^r_{\text{ét}}(S,\sigma_*\Q_l) \quad \text{for all } r.
    \]
    
    Finally, again by the topological invariance property, we have that $\Q_l \simeq (\sigma_* \Q_l)^G$. Hence, by taking $G$ invariants on both sides, we get the desired equality.
\end{proof}

As a consequence of the above computations, we get the following result, which implies that \emph{the Betti numbers of $S$ are the same as those of $E \times Y$.}

\begin{corollary}
\label{cohomology_quotient}
Let $G$ be any finite group scheme acting on a variety $T$ and let $\sigma \colon T \rightarrow S$ be the quotient by $G$. Assume that $S$ exists in the category of schemes. Then
\[
 H^r_{\text{ét}} (S,\Q_l) = H^r_{\text{ét}} (T, \Q_l)^G \quad \text{for all } r.
 \]
\end{corollary}

\begin{proof}
 Thanks to the connected-étale exact sequence of $G$ recalled in (\ref{etale_connected}), we have a factorisation
 \begin{center}
     \begin{tikzcd}
         \sigma \colon T \arrow[rr, "\varphi"] && U \arrow[rr, "\psi"] && S,
     \end{tikzcd}
 \end{center}
 where $\varphi$ and $\psi$ are respectively the quotient by $G^0$ and by $\pi_0(G)$. Thus, it suffices to apply \Cref{etale_finite} to the morphism $\psi$ and then \Cref{etale_infinitesimal} to the morphism $\varphi$ in order to get isomorphisms
 \[
 H^r_{\text{ét}} (S,\Q_l) = H^r_{\text{ét}} (U, \Q_l)^{\pi_0(G)} = H^r_{\text{ét}} (U, \Q_l)^G 
 = H^r_{\text{ét}}(T,\Q_l)^G,
 \]
 and we conclude.
\end{proof}

\begin{lemma}
\label{lifting_curve}
    Let $Y$ be a smooth projective curve over $k$. There exists a smooth complex projective curve $Y_0$ such that
    \[
        H^i_{\text{ét}}(Y,\Q_l) = H^i_{\text{sing}}(Y_0,\Q_l) \text{ for all }i, 
    \]
    where $H^i_{\text{sing}}$ denotes the singular cohomology. In particular, $g(Y) = g(Y_0)$.
\end{lemma}

\begin{proof}
    Let $A\defeq W(k)$ be the ring of Witt vectors. In particular, $A$ is a discrete valuation ring of characteristic $0$ with residue field $k$. Let $s$ be the closed point of $\Spec A$; by \cite[Exposé III, Théorème 7.3]{SGAI}, there exists a smooth proper morphism $\mathcal{Y}\to \Spec A$ such that $\mathcal{Y}_{s} = Y$, and the fiber $\mathcal{Y}_{\eta}$ above the generic point $\eta$ is a smooth projective curve $Y_0$ over $\Frac A$. By the Lefschetz principle, $Y_0$ is defined over a subfield $k_0\subset \Frac A$, which can be embedded in $\C$. So now we can assume that $Y_0$ is a complex smooth projective curve. Next, by \cite[VI, Corollary 4.2]{Milne}, it follows that \[
    H^i_{\text{ét}}(Y,\Q_l) = H^i_{\text{ét}}(Y_0,\Q_l).\] Since $Y_0$ is complex, the Comparison Theorem - see \cite[Theorem 2]{Artin} - gives that 
    $H^i_{\text{ét}}(Y_0,\Q_l)= H^i_{sing}(Y_0,\Q_l)$. 
\end{proof}

\begin{proposition}
\label{Betti}
    Let $S= E \times^G X$ with $G$ any finite subgroup scheme of $E$ and $X$ a $G$-normal curve. Then the Betti numbers of $S$ are as follows:
    \[
    b_1(S) = 2+2g(Y) \quad \text{and} \quad b_2(S) = 2+4g(Y).
    \]
\end{proposition}

\begin{proof}
    Thanks to \Cref{cohomology_quotient} applied to the map $q$, in order to compute the étale cohomology of $S$ it is enough to compute the $G$-invariant part of the étale cohomology of $E \times X$. More precisely, we aim to compute the dimensions of $H^i_{\text{ét}}(E\times X,\Q_l)^G$. Let us start by the computation of $b_1$: by Künneth's formula, together with the fact that both $E$ and $X$ are connected,
    $$
    H^1_{\text{ét}}(E\times X,\Q_l)^G = H^1_{\text{ét}}(E,\Q_l)^G \oplus H^1_{\text{ét}}(X,\Q_l)^G.
    $$
    
    Let us consider the first term on the right hand side: by applying \Cref{lifting_curve} to the smooth curve $F \defeq E/G$, there exists a smooth elliptic curve $F_0$ over $\C$ such that
    \begin{align}
    \label{comparison_E}
    H^1_{\text{ét}} (E,\Q_l)^G = H^1_{\text{ét}}(F,\Q_l)= H^1_{\text{sing}} (F_0, \Q_l) \simeq H_{1,\text{sing}}(F_0,\Q_l)^\vee,
    \end{align}
    where the last isomorphism comes from Poincaré duality. The latter has dimension $2$, because $F_0$ is a complex elliptic curve hence topologically it is a complex torus. Moving on to the second term, let us apply \Cref{cohomology_quotient} to the morphism $\pi \colon X\rightarrow Y$ and then apply \Cref{lifting_curve} to the smooth curve $Y$. This yields that there is a smooth curve $Y_0$ over $\C$, with same genus as $Y$, such that
    \begin{align}
    \label{comparison_curveY}
    H^1_{\text{ét}} (X,\Q_l)^G = H^1_{\text{ét}} (Y,\Q_l) = H^1_{\text{sing}} (Y_0,\Q_l) \simeq  H_{1,\text{sing}} (Y_0,\Q_l)^\vee.
    \end{align}
    The last term, for the same reason as for the elliptic curve above, has dimension $2g(Y)$. Thus, we get the desired value for the first Betti number.

    Moving on to $b_2(S)$, the Künneth's formula together with \Cref{cohomology_quotient} applied to the quotient morphism $\pi$ yields
    \[
    H^2_{\text{ét}} (E \times X, \Q_l)^G = H^2_{\text{ét}} (Y,\Q_l) \oplus \left( H^1_{\text{ét}} (F,\Q_l) \otimes H^1_{\text{ét}} (Y,\Q_l) \right) \oplus H^2_{\text{ét}}(F,\Q_l).
    \]
    Concerning the middle term, taking $G$-invariants commutes with the tensor product is due to the fact that $G \subset E$ acts on $E$ by translation and that $E$ is connected, hence the $G$-action on $H^1_{\text{ét}}(E,\Q_l)$ is trivial. Thanks to Poincaré duality, together with the isomorphisms (\ref{comparison_curveY}) and (\ref{comparison_E}), computing dimensions yields
    \[
    b_2(S) = b_2(E \times X)^G = b_0(Y) + b_1(F_0)b_1(Y_0) + b_0(F) = 2+4g(Y)
    \]
    and we are done.
\end{proof}

\subsection{Dualizing sheaf and Kodaira dimension} 
\label{sec:dualizing}
Let $Y$ be a locally Noetherian scheme and let $g \colon X \rightarrow Y$ be a quasi-projective morphism which is a locally of complete intersection, i.e.\ which factors through a scheme $Z$ into a regular embedding $i$ followed by a smooth morphism. Then the \emph{canonical sheaf} of $g$ is defined as 
    \[
    \omega_{X/Y} := \det (\mathcal{C}_{X/Z})^\vee \otimes i^* (\det \Omega_{Z/Y}^1),
    \]
    where $\mathcal{C}$ is the conormal sheaf. If $Y$ is a smooth variety of dimension $d$ over $k$, its canonical sheaf can just be defined as being the sheaf of regular $d$-forms.

    \medskip

Since in our context we deal with non-smooth $G$-curves, it is convenient to use the following object, which plays the analogous role of the canonical sheaf and generalizes it. The $r$-th \emph{dualizing sheaf} of a proper morphism $g \colon X\longrightarrow Y$, with fibers of dimension $\leq r$, is a quasi-coherent sheaf $\omega_g$, equipped with a canonical isomorphism
    \[
    g_\ast Hom_{\mathcal{O}_X} (\mathcal{F}, \omega_g) \simeq Hom_{\mathcal{O}_Y}(R^rg_* \mathcal{F},\mathcal{O}_Y),
    \]
    for all quasi-coherent $\mathcal{O}_X$-modules $\mathcal{F}$; for more details on this topic, see \cite[\S 6.4]{Liu}.\\
    
 When it is defined, we denote as $\omega_T$ the dualizing sheaf of the structure morphism $T/ k$, which in particular is the same as the canonical sheaf when $T$ is smooth. The dualizing sheaf satisfies the following adjunction formula: for $g\colon X\rightarrow Y$ a flat projective l.c.i.\ morphism of pure relative dimension $r$, the $r$-th dualizing sheaf $\omega_g$ is isomorphic to the canonical sheaf $\omega_{X/Y}$. Moreover, for $Y \rightarrow Z$ flat, projective, of pure relative dimension $r$ and l.c.i, we have
        \begin{align}
        \label{adjunction}
        \omega_{X/Z} = \omega_{X/Y} \otimes g^*\omega_{Y/Z}.
        \end{align}

\begin{remark}\label{dualizing_existence}
    The following results on the dualizing sheaf can be found in \cite[\S 6]{Liu} (see e.g. Corollary 4.29 and Theorem 4.32) or in \cite[III. \S 7]{Hartshorne} (see e.g. Proposition 7.5, Theorem 7.6 and Corollary 7.7). For every projective variety, the dualizing sheaf exists. As a direct consequence of the fact that every $G$-normal curve $X$ is locally of complete intersection, see \Cref{lci}, the dualizing sheaf $\omega_X$ is invertible. Its degree is defined as the integer
    \[
    \deg(\omega_X) \defeq \chi(\omega_X) - \chi(\mathcal{O}_X).
    \]
    Moreover, $X$ is Cohen-Macaulay; hence, we have the Serre duality on $X$ (see \cite[III. Corollary 7.7]{Hartshorne}); i.e., for any locally free sheaf $\mathcal{F}$ on $X$:
    \[
        \mathrm{H}^1(X,\mathcal{F}) \simeq \mathrm{H}^0(X,\mathcal{F}^\vee \otimes \omega_X)^\vee.
    \]
\end{remark}

The morphism $\pi$ being flat and a local complete intersection, the relative dualizing sheaf of $\pi$ is isomorphic to the relative canonical sheaf $\omega_{X/Y}$; in particular, it is equipped with a $G$-linearization (see \cite[Section 7]{diagos}). This implies that the dualizing sheaf $\omega_X$ is equipped with natural $G$-action, which extends the $G$-action on the canonical sheaf on the smooth locus of $X$.

\begin{lemma}
\label{kappa}
    Let $S$ be the surface as in \Cref{diagram}. 
    Then the canonical sheaf is described in terms of the dualizing sheaf of $X$ as follows:
    \begin{align*}
    \omega_S = (q_*)^G {pr_X}^*\omega_X.
    \end{align*}
\end{lemma}

\begin{proof}
    By \cite[Lemma 7.1]{diagos}, we have that the relative canonical sheaf of $q$ is trivial. Hence,
    \[
    q^*\omega_S = \omega_{E \times X} = {pr_E}^*(\omega_E) \otimes {pr_X}^*(\omega_X) = {pr_X}^*(\omega_X).
    \]
    Taking the push-forward via $q_*$ and taking $G$-invariants then yields
    \[
    \omega_S = (q_*)^G q^*\omega_S = (q_*)^G {pr_X}^*(\omega_X)
    \]
    and we are done.
\end{proof}

\label{sec:kodaira}
The guiding idea of this section is to show that the Kodaira dimension of $S$ should be \emph{the same as the one of $X$}; since $X$ is non-smooth in general we need to make use of the notion of dualizing sheaf. If $H^0(X,L) \neq 0$, we denote by $\Phi_{|L|}\colon X\dashrightarrow \proj^N$ the rational map induced by $L$.

\begin{definition}
Let $L$ be a line bundle on $X$ and
\[
N(L) \defeq \{n\geq 1,\ H^0(X,L^n) \neq 0\}.
\]
The \emph{Iitaka dimension} $\kappa(X,L)$ of $L$ on $X$ is defined as $-\infty$ if $N(L) = \emptyset$, else as the quantity
\[
  \max_{n \in N(L)} \left( \dim \Phi_{\vert nL\vert}(X) \right).
\]
\end{definition}

If moreover $X$ is smooth, then the canonical sheaf $\omega_X$ on $X$ exists and the \emph{Kodaira dimension} of $X$ is the Iitaka dimension of $\omega_X$, which provides a birational invariant for smooth varieties. 
For a $G$-normal curve $X$, we denote as
\[
\kappa^\prime(X) \defeq \kappa(X,\omega_X)
\]
the Iitaka dimension of its dualizing sheaf, which is well defined thanks to \Cref{dualizing_existence}. This notion coincides with the classical one of Kodaira dimension when $X$ is smooth, but it can be different when $X$ is not a normal curve; see e.g.\ \cite[Example 2.9]{Chen_Zhang}. 

\begin{lemma}
\label{equalityKodaira}
The canonical ring 
\[
R(S,\omega_S) \defeq \bigoplus_{n=0}^\infty H^0(S,\omega_S^n)
\]
is a finitely generated $k$-algebra, isomorphic to the $G$-invariant part of $R(X,\omega_X)$. Moreover, the equality $\kappa(S) = \kappa^\prime(X)$ holds.
\end{lemma}

\begin{proof}
    As noticed in the proof of \Cref{kappa}, we have the equality $q^*\omega_S = pr_X^*\omega_X$. Taking on both sides the $n$-th tensor powers, then the direct image $q_*$ and using that $q$ is a $G$-torsor, we obtain that 
    \[
        \omega_S^{n} = q_*(pr_X^*(\omega_X)^{n})^G.
    \]
    Finally, by taking the global sections, we get:
    \begin{align}
    \label{sections}
H^0(S,\omega_S^n) = H^0(E\times X,pr_X^*(\omega_X^{n}))^G = H^0(X,\omega_X^n)^G, \quad \text{for all } n.
\end{align}
The ring of sections $R(X,\omega_X)$ of the dualizing sheaf is a finitely generated $k$-algebra, because $X$ is a projective curve; moreover, it is integral over
\[
R(S,\omega_S) = R(X,\omega_X)^G \subset R(X,\omega_X).
\]
Thus, we can conclude that $R(S,\omega_S)$ is also finitely generated, and that the two Iitaka dimensions coincide.
\end{proof}

\subsection{Proofs of \Cref{TheoremA} and \Cref{CorollaryB}}

\begin{proof}[Proof of \Cref{TheoremA}]
    By Lemmas \ref{lem:minimal} and \ref{equalityKodaira}, $S$ is a relatively minimal surface such that $\kappa(S) = \kappa'(X)$. The computations of the Betti numbers of $S$ are given in \Cref{Betti}.

    \smallskip

    Next we prove the table of classification. Assume $\kappa(S) = -\infty$. The surface $S$ is not isomorphic to the projective plane, so $S$ is equipped with a structure of $\proj^1$-bundle and has Picard rank two. Since the morphism $f\colon S\to Y$ has general fiber $E$, the structural morphism of $\proj^1$-bundle is the fibration $h\colon S\to E/G$. Hence $X=\proj^1$, this implies $Y=\proj^1$ and we deduce the invariants of $S$. Conversely, if $X = \proj^1$, then $h$ is a $\proj^1$-bundle over the elliptic curve $E/G$, which is isogenous to $E$, so $\kappa(S) = -\infty$.

    \smallskip
    
    Assume now that $\kappa(S)=0$. Then $X$ is a curve with arithmetic genus
    \[
    p_a(X) \leq \kappa'(X) + 1 = 1,
    \]
    where $\kappa'(X)$ is the Iitaka dimension of the dualizing sheaf on $X$. Now notice that the morphism $X\to Y$ factorizes through the quotient $X\to X/G^0$ and $X/G^0$ is a smooth curve with geometric genus $g(X/G^0) = g(X)$ (see \cite[Corollary 4.8 and Remark 5.2]{Brion2}). This implies that
    \[g(Y)=p_a(Y) \leq g(X) \leq p_a(X),\]
    where the second inequality follows from \cite[IV. Ex 1.8]{Hartshorne}. 
    
    Therefore, $Y=\proj^1$ or an elliptic curve, and $p_a(X)\in \{0,1\}$. If $p_a(X)=0$, then $X=\proj^1$ as $X$ is an irreducible curve (see again \cite[IV. Exercise 1.8]{Hartshorne}). Then $h$ is again a $\proj^1$-bundle, which contradicts that $\kappa(S)=0$. Hence $p_a(X) = 1$, so $X$ is either an elliptic curve, or a rational curve with a cusp or a node (again by \cite[IV. Exercise 1.8]{Hartshorne}). By \cite[Corollary 4.6]{Brion2}, the latter case is prohibited because $G$-normal curves admit only cusps as singularities. It suffices to see that the cases where $X$ is a rational curve with a cusp or an elliptic curve correspond to the surfaces of Kodaira dimension zero.
    
    Now, if $X$ is a rational curve with a cusp, then $Y=\proj^1$, so $q(S)=1$, $b_1(S) = b_2(S) = 2$, and $S$ is a quasi-hyperelliptic surface: see \cite[Proposition p.\ 26]{BombieriMumford}. If $X$ is an elliptic curve, then we distinguish two cases: either $G$ acts by translations on $X$, in which case $Y$ is an elliptic curve and $S$ is an abelian surface; or $G$ acts on $X$ not only by translations, $Y=\proj^1$ and $S$ is a hyperelliptic surface: see \cite[Theorem 4]{BombieriMumford}. This proves that if $S$ is a properly elliptic surface, then $X$ is a $G$-normal curve which was not considered above (i.e.\ $X$ is not isomorphic to $\proj^1$, a rational curve with a cusp or an elliptic curve); and any of these $G$-normal curves gives rise to a surface of Kodaira dimension one.
\end{proof}

\begin{proof}[Proof of \Cref{CorollaryB}]
    By \cite[Proposition 3.24]{fong_1}, $\Aut^0(S)$ is an abelian variety; and if $S$ is an abelian surface, then $\Aut^0(S)=S$ acts on itself by translations. To extend the classification of pairs $(S,\Aut^0(S))$ in positive characteristic, where $S$ is relatively minimal with $\kappa(S)\geq 0$, it suffices to determine the surfaces $S$ for which $\Aut^0(S)$ is an elliptic curve (see \cite[Remark 3.26]{fong_1}). Such surfaces are isomorphic to a contracted product $E\times^G X$, where $E$ is an elliptic curve, $G\subset E$ a finite subgroup scheme and $X$ a $G$-normal curve. 
    
    Assume that $\Aut^0(S)$ is an elliptic curve. By \Cref{TheoremA}, if $\kappa(S)=0$, then $S$ is a quasi-hyperelliptic surface or a hyperelliptic surface. Conversely, every quasi-hyperelliptic and hyperelliptic surface is a contracted product $S=E\times^G X$ equipped with a faithful action of $E=\Aut^0(S)$. If $\kappa(S)=1$, then $\Aut^0(S)$ is an elliptic curve if and only if $S$ is also a contracted product of that form. Finally, in all remaining cases from the classification of surfaces, we obtain that $\Aut^0(S)$ is trivial.
\end{proof}

%-------------------------------------------------------------------------------------------

\section{The diagonalizable case: on the multiple fibers and the Picard scheme}\label{Section:diagonalizable}

From now on, we keep the notation of \Cref{diagram}, and assume moreover the group $G$ to be a finite \emph{diagonalizable} subgroup scheme of $E$, unless explicitly stated otherwise. Under this assumption, the elliptic fibration $f$ satisfies the following: all multiple fibers are tame, and we have a Hurwitz formula for the $G$-normal curve $X$.

\subsection{Multiplicity of fibers}
Let us keep the notation of the diagram in \Cref{diagram} and denote as 
\[
f^{-1}(y) = m(y) \cdot f^{-1}(y)_{\text{\text{red}}},\]
the schematic fiber over the point $y \in Y$, seen as a divisor in $S$, with $m(y)$ its multiplicity. We assume in this part that $E$ is ordinary, so that $G$ is a finite diagonalizable group. Let us recall that we denote as $H(x)$ the stabilizer of the action at the point $x \in X$ and as $n(x)$ its order. 

We distinguish two types of multiple fibers for (quasi)-elliptic surfaces:

\begin{definition}\cite{BombieriMumford}
    A fiber above $y \in Y$ is called a \emph{wild} fiber if $\dim_k \mathcal{O}(f^{-1}(y)) \geq 2.$ Otherwise, it is called a \emph{tame} fiber.
\end{definition}

In order to study multiple fibers, it is useful to look at the quotient map 
\[
\pi \colon X \longrightarrow Y = X/G.
\]

\begin{lemma}
\label{trivialR1} 
    Let us keep the above assumptions and let $G$ be a finite diagonalizable group. Then $R^1f_*\mathcal{O}_S = \mathcal{O}_Y$.
\end{lemma}

\begin{proof}
    By definition $\mathcal{O}_S = (q_*)^G\mathcal{O}_{E\times X})$. Since $G$ is diagonalizable and $q_*$ exact, 
    \[
    R^1f_*(\mathcal{O}_S) = R^1f_*(q_*)^G\mathcal{O}_{E\times X}=(R^1f_*(q_*(\mathcal{O}_{E\times X})))^G = (R^1(fq)_*\mathcal{O}_{E\times X})^G.
    \]
    Let $\phi= fq$. Taking $U\subset Y$ affine and $V\defeq \pi^{-1}(U)$, we get that
    \begin{align*}
        (R^1\phi_*\mathcal{O}_{E\times X}(U))^G & = H^1(\phi^{-1}(U),\mathcal{O}_{E\times X})^G  = H^1(E\times V,\mathcal{O}_{E\times V})^G \\
        & = (H^1(E,\mathcal{O}_E)\otimes \mathcal{O}(V))^G = \mathcal{O}(U).
    \end{align*}
    Thus $R^1f_*(\mathcal{O}_S) = \mathcal{O}_Y$ and we find that there are no wild fibers.
    \end{proof}

\begin{proposition}
\label{nowild}
    With the same assumption as in \Cref{trivialR1}, the following holds: for every $x \in X$ and $y\defeq \pi(x)$, the multiplicity $m(y)$ is equal to $n(x)$. 
\end{proposition}

\begin{proof}
    Recall that $H=H(x)$ is the stabilizer of $G$ at the point $x$; let $Z\defeq X/H$. Then we get the following commutative diagram:
    \[
    \begin{tikzcd}
        X \arrow[rr,"\pi"]\arrow[rd,"\pi_H",swap] && Y\\
        & Z\arrow[ru,"\pi_{G/H}",swap] &.
    \end{tikzcd}
    \]
    First, the morphism $\pi_{G/H}$ is a $G/H$-torsor over an open neighborhood of $y$; this holds by the local description of \Cref{sec:local}. Next, let us consider the restriction of $\pi_H$ to the fibers over the point $y$, which gives a $G$-equivariant morphism
   \[
   X_y := \pi^{-1}(y) \longrightarrow \pi_{G/H}^{-1} (y) \simeq G/H.
   \]
   Thus we can write the fiber of $\pi$ over $y$ as the following contracted product:
   \begin{align*}
      X_y = G\times^H X_z,
   \end{align*}  
   where $X_z$ is the fiber of $\pi_H$ over the point $z = \pi_H(x)$. Taking the preimages of $q$ on both sides, then taking the quotient by $G$ gives 
    \begin{align}
    \label{fiber_contracted}
    f^{-1}(y) = (E\times X_y)/G = E\times^H X_z,
    \end{align}
    with reduced subscheme being $E/H$. Thus, the regular functions on the fiber satisfy
    \[
        \mathcal{O}(f^{-1}(y)) = \mathcal{O}(E\times X_z)^H = \mathcal{O}(X_z/H) = k,
    \]
    where we use again that $H$ is linearly reductive; this implies that every multiple fiber is tame. The fact that the multiplicity is equal to $n(x)$ follows again from \Cref{sec:local}.
\end{proof}

\subsection{Dualizing sheaf formula}

\begin{theorem}\label{thm:dualformula} \textnormal{(Dualizing sheaf formula)}\\
    Let $G$ be diagonalizable. Then the quotient morphism $\pi \colon X \rightarrow Y$ is flat, locally of complete intersection and we have an isomorphism of $G$-linearized sheaves
    \begin{align}
    \label{formula:dualizing}
    \omega_X \simeq (\pi^*\omega_Y) \otimes \mathcal{O}_X \left( \sum ( n(x) -1) \, G \cdot x\right)
    \end{align}
    where the sum is taken over the $G$-orbits of rational points of $X$.
\end{theorem}

The above fundamental result is \cite[Corollary 7.1]{diagos}; a way to reformulate it is as follows:
\begin{align}
\label{ramification_canonical}
\omega_X \otimes \mathcal{O}_X(G \cdot \Delta_X) = \pi^* (\omega_Y \otimes \mathcal{O}_Y(\Delta_Y))
\end{align}
where $\Delta_X = X \setminus X_{\text{fr}}$ is the divisor of the non-free points and $\Delta_Y = Y \setminus Y_{\text{fr}}$ is the branch divisor. 

%As a consequence of Theorem \ref{thm:dualformula}, the dualizing sheaf is isomorphic to the canonical sheaf $\omega_{X/Y}$; in particular, it turns out to be invertible (see \cite[\S 6.4, Theorem 4.32]{Liu}).

\begin{corollary}
\label{Hurwitz}
    The degree of the dualizing sheaf of $X$ can be calculated as follows
    \[
    \deg \omega_X = \deg \omega_Y \cdot \vert G \vert + \sum_{y \in Y} (\vert G \vert - [G \colon H(x)]),
    \]
    where $[G \colon H(x)]$ only depends on the image $y = \pi(x)$.
\end{corollary}

\begin{corollary}
\label{cor:degreeomega}
    We have that $\kappa (S) = 1$ if the degree of $\omega_X$ is positive, that $\kappa(S) = 0$ if $\omega_X$ is trivial, and that $\kappa(S) <0$ if the degree of $\omega_X$ is negative.
\end{corollary}

Let us apply the above Corollary to deduce information on the Kodaira dimension $\kappa(S)$ from the genus of the smooth curve $Y$ and from the number of multiple fibers of $f$. 
Let us notice that the Kodaira dimension of $S$ \emph{cannot be $2$}, since the self-intersection of $\omega_S$ is never strictly positive in our context.

\medskip

\textbf{Higher genera}: if $Y$ is of genus at least equal to two, then $\omega_Y$ is ample, hence the Kodaira dimension of $S$ is equal to one. 

\begin{remark}
\label{rem:michel_last}
    Let us notice that such actions actually exist: let us fix some $g\geq 1$ and some $r \geq 1$ and consider the action of the trivial group on a smooth projective curve of genus $g$. By applying \cite[Remark 9.4]{diagos} to this case, we get the following: for any genus $g$ and for any $r$, there exists a $\boldsymbol{\mu}_{p^r}$-normal projective curve $X_{g,r}$ with geometric genus $g$.
\end{remark}

\textbf{Genus one}: if $Y$ is an elliptic curve, then $\omega_Y$ is trivial, thus there are two different cases to deal with. In the first case, there is at least one multiple fiber, which implies that the degree of $\omega_X$ is strictly positive, and thus by (\ref{kappa}) we once again get $\kappa(S) = 1$. In the second case, we have no multiple fibers which means that $\pi$ is a $G$-torsor. In this case, we now show that $X$ is also an elliptic curve; for example, when $G= \boldsymbol{\mu}_{p^r}$, then $X$ is isomorphic to the quotient of $Y$ by the constant group $\Z/p^r\Z$ which is the Cartier dual of $G$.

\begin{lemma}\label{lem:torsoroverelliptic}
    A nontrivial reduced $\boldsymbol{\mu}_{p^r}$-torsor over an elliptic curve is itself an elliptic curve.
\end{lemma}

\begin{proof}
    Let $Y$ be an elliptic curve and $\pi\colon X\to Y$ a $\boldsymbol{\mu}_{p^r}$-torsor such that $X$ is reduced. We denote by $F^r \colon X\to X^{(p^r)}$ the $r$-th iterated (relative) Frobenius morphism of $X$. Then by the universal property of the quotient, there exists a unique morphism $\varphi \colon Y\to X^{(p^r)}$ such that the following diagram commutes
    \[
        \begin{tikzcd}
            \widetilde{X} \arrow[dd, "\tau"] \arrow[rr,"\pi \circ \tau"] && Y \arrow[dd, "\varphi"] \\
            &&\\
            X \arrow[rruu, "\pi"] \arrow[rr, "F^r"] && X^{(p^r)}
        \end{tikzcd}
    \]
    where $\tau \colon \widetilde{X} \rightarrow X$ is the normalization morphism. Since $\varphi$ is a birational morphism and $Y$ is smooth, it follows that $\varphi$ is the normalization map. In particular, the geometric genus of $X^{(p^r)}$ is equal to one; hence, the same holds for $X$ by \cite[IV, Proposition 2.5]{Hartshorne}. Then $\widetilde{X}$ is an elliptic curve. 
    It remains to show that $X$ is smooth. Moreover, $\tau$ being an isomorphism at the generic point, the generic fiber of the morphism $\pi \circ \tau$ is $\boldsymbol{\mu}_{p^r}$. Moreover, up to automorphism of $Y$, we can assume that $\pi \circ \tau$ is an isogeny of elliptic curves, i.e., has finite kernel $\boldsymbol{\mu}_{p^r}$. Thus $\tau$ is an isomorphism and $X$ is also an elliptic curve. 
\end{proof}

The following provides a family of examples (over an elliptic curve) with an arbitrarily large number of multiple fibers, all of Kodaira dimension one.

\begin{lemma} 
Let $p \geq 5$ and consider the following elliptic curve
\[
Y = \{ g(x,y,z) = y^2z- x(x+z)(x-z) = 0\} \subset \proj^2_{x,y,z}
\]
embedded as a smooth cubic in the projective plane.
Then we define the curve 
\[
\boldsymbol{\mu}_{p^n} \curvearrowright X \defeq \{w^{p^n} = zh(x,z) + y^{p^n}, \, g =0 \} \subset \proj^3_{x,y,z,w} \longrightarrow Y,
\]
the $\boldsymbol{\mu}_{p^n}$-action being by multiplication on the $w$-coordinate. We assume moreover that 
\[
h(x,z) = \prod_{i=2}^{p^n} (a_ix+z)
\]is a homogeneous polynomial of degree $p^n-1$, such that $a_i \in k \setminus \{0,1,-1\}$ are all distinct and such that their sum equals zero. We claim that $X$ is singular at exactly three points, all belonging to $X_{\text{fr}}$, and that it is $\boldsymbol{\mu}_{p^n}$-normal.  The number of multiple fibers is $p^n$ and this provides a family of examples (over an elliptic curve) with an arbitrarily large number of multiple fibers, all of Kodaira dimension one.
\end{lemma}

\begin{proof}
We denote as $(-)_x, \, (-)_y$ and $(-)_z$ the respective partial derivatives. Let us explicitly compute the singular points of $X$: both derivatives with respect to $w$ vanish, while taking derivatives with respect to $x,y$ and $z$ yields the conditions $2yz=0$ and
\begin{align*}
    &  z h_x= \lambda( 3x^2-z^2),\\
    & h(x,z) + zh_z = \lambda(-y^2-2xz),
\end{align*}
for some $\lambda \in k$. In particular, either $y$ or $z$ must vanish. If $z=0$, then using the equations of $X$ we have that $x$ also vanishes, so the point we have to consider is $\underline{x}_0 = (x,y,z,w) = (0\colon 1\colon 0 \colon 1)$, which is indeed singular and a free point for the $\boldsymbol{\mu}_{p^n}$-action. 

Next, assuming $y=0$ gives the three points $\underline{x} = (0\colon 0 \colon 1 \colon 1)$, $\underline{x}^\prime= (1\colon0\colon 1\colon \alpha)$ and $\underline{x}^{\prime\prime}=(1\colon 0\colon -1\colon \beta)$, with $\alpha$, $\beta \in k^\times$ determined by the equation involving $w$. All three are free points for the $\boldsymbol{\mu}_{p^n}$-action, because the non-free points have the first and third coordinates corresponding to zeros of the polynomial $zh(x,z)$. At the first point, computing the derivatives with respect to $x$ gives $h_x (0,1) \neq 0$, which is a contradiction with the assumption that the sum of the $a_i$ vanishes; hence, $\underline{x}$ is not a singular point. 
Next, computing derivatives at the second and third point yields that they are singular respectively if and only if 
\begin{align}
\label{calcoli}
h_x (1,1) + h(1,1) + h_z(1,1) = 0 \quad \text{and} \quad h_x(1,-1) + h(1,-1) -h_z(1,-1) = 0
\end{align}
An explicit computation yields that
\begin{align*}
    h(1,1) & = \prod_{i=2}^{p^n} (a_i+1),\\
    h_x(1,1) & = \sum_{i=2}^{p^n} a_i \prod_{j\neq i} (a_j+1) = -\prod_{i=2}^{p^n} (a_i+1) \left( 1+\sum_{j=2}^{p^n} \frac{1}{a_j+1}\right),\\
    h_z(1,1) & = \prod_{i=2}^{p^n} (a_i+1) \sum_{j=2}^{p^n} \frac{1}{a_j+1},
\end{align*}
which means that the first equality of (\ref{calcoli}) is satisfied and thus that $\underline{x}^\prime$ is a singular point. An analogous computation shows that the same holds for $\underline{x}^{\prime\prime}$. Since the quotient $Y$ is smooth, and since all the singular points are free for the $\boldsymbol{\mu}_{p^n}$-action, by \Cref{lci} the curve $X$ is indeed $\boldsymbol{\mu}_{p^n}$-normal.
\end{proof}

\textbf{Genus zero}: finally, let us assume that we are over $Y= \proj^1$.

\begin{lemma}
    If $G$ is infinitesimal, then there must be at least \emph{two} multiple fibers.
\end{lemma}

\begin{proof}
    First, let us show that  $\pi$ cannot be a $G$-torsor. Let us recall that by Kummer theory, the following 
    \[
    0 \longrightarrow {\boldsymbol{\mu}}_n \longrightarrow \Gm \stackrel{(\cdot)^n}{\longrightarrow} \Gm \longrightarrow 0
    \]
    is a short exact sequence for the \textit{fppf} topology, over any scheme $Z$. This can be interpreted by saying that the data of a  ${\boldsymbol{\mu}}_n$-torsor over $Z$ is the same as the data of a line bundle $L$ over $Z$, together with a section $\sigma \colon L^{\otimes n} \stackrel{\sim}{\longrightarrow} \mathcal{O}_Z$ trivializing its $n$-th power. Hence, a ${\boldsymbol{\mu}}_{p^r}$-torsor over the projective line is a line bundle over $\proj^1$ whose $p^r$-th tensor power is trivial, which does not exist. Next, we also want to exclude that there is only one multiple fiber. By assuming the action to be generically free, one has that $G= {\boldsymbol{\mu}}_{p^r}$ for some $r$; in particular, again by the Kummer sequence, this would give a $G$-torsor over the affine line $\AA^1$. The latter is affine and its invertible functions are non-zero constants; thus, it does not admit any non-trivial $G$-torsor.
\end{proof}

When considering elliptic surfaces over $\proj^1$, in negative Kodaira dimension we have ruled surfaces, discussed above in \Cref{section:ruled}, while if $p\leq 3$, the class of quasi-hyperelliptic surfaces provides examples of Kodaira dimension zero. Thus, we now focus on showing that we \emph{very often} get a surface of Kodaira dimension one, namely as soon as the number of multiple fibers is big enough.

\begin{lemma}
\label{Nfibers}
Assumption as above; let $G$ be infinitesimal and $Y= \proj^1$.\\
If one of the following conditions is satisfied, then $\kappa(S) = 1$:
\begin{itemize}
    \item $p \geq 5$ and there are at least $3$ multiple fibers;
    \item $p \geq 3$ and there are at least $4$ multiple fibers;
    \item there are at least $5$ multiple fibers.
\end{itemize}
\end{lemma}

\begin{proof}
    We can assume that $G= \boldsymbol{\mu}_{p^r}$ for some $r \geq 1$; hence for any non-free point $x\in X$ we can denote as $p^{s(x)}= n(x)$ the order of the stabilizer $H(x)$. Let $N$ be the number of multiple fibers. Then
    \begin{align*}
        \deg \omega_X & = -2p^r + \sum_{x \in X\setminus X_{\text{fr}}} (p^r-p^{r-s(x)})  = -2p^r + \sum_{s(x) \geq 1} (p^r-p^{r-s(x)})  \\
        & \geq -2p^r + N(p^r-p^{r-1}) = p^{r-1}(p(N-2)-N)
    \end{align*}
    is strictly positive if one of the three above conditions are satisfied. Thus we can conclude by \Cref{cor:degreeomega}.
\end{proof}

\begin{example}
\label{ex_multiplication}
    Let us consider any $p$ and the action of $\boldsymbol{\mu}_p$ on $X= \proj^1$ by multiplication on the first homogeneous coordinate. This action has two fixed points, namely $x=0$ and $x= \infty$ such that $s(0) = s(\infty) = 1$, where $p^{s(x)} = n(x)$ is the order of the stabilizer at the point $x$. Then we get
    \[
    \deg \omega_X = -2p + (p-p^{1-s(0)})+(p-p^{1-s(\infty)}) = -2p + 2(p-1) = -2
    \]
    which gives a negative Kodaira dimension. 
\end{example}

\begin{example}
\label{example2}
The following class of examples is already mentioned in \cite{Brion2}: let us consider the group $G= \boldsymbol{\mu}_{p^r}$ acting on the projective plane by multiplication on the $z$-coordinate. This action stabilizes the curve
 \[
 X= \{ z^{p^r} = f(x,y)\} \subset \proj^2,
 \]
 where $f$ is a homogeneous polynomial of degree $p^r$ with pairwise distinct roots. The curve $X$ has then exactly $p^r$ fixed points. In particular, if $p^r \geq 5$, by \Cref{Nfibers} this provides an infinite family (with arbitrary large number of fibers) of elliptic fibrations with target $\proj^1$ and with Kodaira dimension $1$. All of them, as we will see in \Cref{TheoremA}, are properly elliptic surfaces. On the other hand, a case in which we get trivial Kodaira dimension is when $p^r=3$, in characteristic three, for which
 \[
 \deg \omega_X = -2 \cdot 3+ 3(3-1) = 0,
 \]
 Let us also notice that, if $p=3$ and there are at least three multiple fibers, by the computation of \Cref{Nfibers} we have that $\kappa(S)$ cannot be strictly negative.
\end{example}

\begin{example}
    Let $p=2$ and consider the action of $G= \boldsymbol{\mu}_2$ on the curve
    \[
    X = \{z^4= g(x,y)\} \subset \proj^2,
    \]
    where $g$ is a homogeneous polynomial of degree $4$ with four distinct roots and $G$ acts on $\proj^2$ by multiplication on the $z$-coordinate again. This action has exactly four fixed points, which yields a zero Kodaira dimension. Notice that, for the same reason as in \Cref{example2}, the Kodaira dimension can never be strictly negative if there are at least four multiple fibers.
\end{example}

    Let us note that when $p=2$, every $\boldsymbol{\mu}_2$-normal variety is uniform, since (with the notation of \Cref{sec:local}) there is only one possible choice both for the subgroup $H= H(x)$ and for the weight $\nu= \nu(x)$; see \cite[Remark 6.5]{diagos} for a more detailed explanation. On the other hand, in \Cref{ex_multiplication} we have a non-uniform $\boldsymbol{\mu}_p$-curve for any $p\geq 3$. Moreover, by \cite[Section 5]{diagos}, there is a bijection between uniform $\boldsymbol{\mu}_n$-normal curves over $\proj^1$ and reduced effective divisors $\Delta$ in $\proj^1$ of degree divisible by $n$. The divisor $\Delta$ corresponds to the divisor of the branch points. In particular, a $\boldsymbol{\mu}_2$-normal curve over the projective line, being necessarily uniform, always has an even number of non-free points.

%There is one case left, namely when $Y=\proj^1$, $G= \boldsymbol{\mu}_{2^r}$ and there are exactly three multiple fibers.
%Let us recall that for $r\geq 2$, if a curve is $\boldsymbol{\mu}_{p^r}$-normal then one cannot deduce that it is $\boldsymbol{\mu}_p$-normal in general, so we cannot reduce to the case above. 

%Let us consider the $G$-action on $\proj^2$ given by 
%\[
%t \cdot (x,y,z) = (x,ty,t^2z),
%\]
%and look for a curve passing through all three coordinate points (which are the three fixed points by the action). The only $\boldsymbol{\mu}_{2^r}$-stable curve I can think of is given by the following equation:
%\[
%x^{2^r-1} y^6 + y^{2^r+4} z + x^{2^r+2} z^3 = 0,
%\]
%because the equation is homogeneous of degree $2^r+5$ and the $G$-action is given by multiplication by $t^6$ on every monomial.

%--------------------------------------------------------------------------------------

\subsection{Irregularity and Euler characteristic} In this section, we focus on computing some invariants of the surface $S$, mainly coming from the study of the $G$-module
\[
H^0(X,\omega_X) = \bigoplus_{\lambda \in X^*(G)} H^0(X,\omega_X)_\lambda,
\]
which decomposes into its $G$-weight spaces. We compute the irregularity $h^1(S,\mathcal{O}_S)$ of the elliptic surface $S$, as well as its Euler characteristic. Let us fix the following notation for the dimension of the weight spaces:
\[
h^0(\omega_X)_\lambda \defeq \dim H^0(X,\omega_X)_\lambda, \quad \text{for } \lambda \in X^*(G).
\]
Let apply once again a result of \cite{diagos} concerning the above integers.  Let $G$ be diagonalizable and $\lambda$ be a character of $G$. If $\lambda$ is nonzero, for any $y \in Y \setminus Y_{\text{fr}}$, let $n(y)$ be the order of the stabilizer $H(y)$ at some $x \in \pi^{-1}(y)$. Moreover, let $\nu(y)$ be the associated weight, as in the local description recalled in \Cref{sec:local} and let $m(y,\lambda)$ be the integer such that $1 \leq m(y,\lambda) \leq n(y)-1$ and the character $\lambda-m(y,\lambda)\nu(y)$ restricts trivially to $H(y)$. Then \cite[Proposition 8.2]{diagos} specializes to the following statement (since here we work over an algebraically closed field $k$).

\begin{lemma}
\label{lem:weightspaces}
   Assumptions as above. Then $h^0(\omega_X)_0 = g(Y)$ and
    \[
    h^0(\omega_X)_\lambda =  g(Y) -1 + \sum_{y \in Y\setminus Y_{\text{fr}}} \left( 1- \frac{m(y,\lambda)}{n(y) }\right) \quad \text{for } \lambda \neq 0.
    \]
\end{lemma}

Let us emphasize that a key ingredient in order to make the computation of \Cref{lem:weightspaces} is the dualizing sheaf formula of (\ref{ramification_canonical}).

\begin{corollary}\label{Irregularity}
Let $G$ be diagonalizable. Then the following hold:
\[
h^0(\omega_S) = h^2(\mathcal{O}_S) = g(Y), \quad \text{and} \quad h^1(\mathcal{O}_S) = h^1(\omega_S) = g(Y)+1.
\]
In particular, the Euler characteristic of $S$ satisfies $\chi(\mathcal{O}_S) = 0$.
\end{corollary}

\begin{proof}
First, let us compute $h^0(\omega_S)$. By (\ref{sections}), we have that
\[
H^0(X,\omega_X)^G = H^0(S,\omega_S);
\]
hence it suffices to compute the dimension of the left hand side term. By \Cref{lem:weightspaces}, we have
\[
h^0(\omega_S) = \dim(H^0(\omega_X)^G) = h^0(\omega_X)_0 = g(Y).
\]

Next, let us move on to the computation of $h^1(\mathcal{O}_S)$. Since $E \times X \to S$ is a $G$-torsor and $G$ is diagonalizable, we have
\[
H^i(S,\mathcal{O}_S) = H^i(E \times X, \mathcal{O}_{E \times X})^G \quad \text{for all } i.
\]
Applying this to $i=1$ and using Kunneth's formula, we get
\begin{align*}
H^1(\mathcal{O}_S) & = H^1(\mathcal{O}_{E\times X})^G = (H^0(\mathcal{O}_E) \otimes H^1(\mathcal{O}_X))^G \oplus (H^1(\mathcal{O}_E)\otimes H^0(\mathcal{O}_X))^G\\
& = (H^0(X,\omega_X)^\vee)^G \oplus k. 
\end{align*}
By taking dimensions on both sides, we get the desired equality.
\end{proof}

%-----------------------------------------------------------------------------------

%---------------------------------------------------------------------------
\subsection{Picard group and Albanese variety}

For a smooth projective surface $S$, the Picard scheme $\underline{\Pic}_S$ can be non-reduced in positive characteristic. For elliptic fibrations, this phenomenon is related to the existence of wild fibers, as explained in \cite{Liedtke_Picard_nonreduced}.
In the case where $G$ is diagonalizable, the surfaces $E\times^G X$ have only tame fibers, as we showed in \Cref{nowild}. The computations of the irregularity and of the Betti numbers give the following result:

\begin{corollary}
\label{Pic0reduced}
    Let $G$ be diagonalizable. Then the Picard scheme of $S=E\times^G X$ is reduced.
\end{corollary}

\begin{proof}
    The vector space $H^1(S,\mathcal{O}_S)$ can be seen as the Lie algebra of $\underline{\Pic}_S$, while
    $b_1(S)/2$ is the dimension of $\underline{\Pic}_S$. By \Cref{Irregularity} and \Cref{Betti}, \[
    h^1(S,\mathcal{O}_S) = 1+g(Y) = b_1(S)/2.\]
    Therefore, the group scheme $\underline{\Pic}_S$ is reduced (see e.g.\ \cite[Section 3.3]{Liedtke}).
\end{proof}

Since we are interested in $G$-varieties, when $G$ is a finite group scheme, let us introduce an equivariant version of the Picard group: for a $G$-variety $X$ and $\pi\colon L \rightarrow X$ a line bundle over $X$, a \emph{$G$-linearization} of $L$ is a $G$-action on $L$ such that $\pi$ is $G$-equivariant and which commutes with the $\Gm$-action by multiplication on the fibers. The tensor product of two $G$-linearized line bundles is also $G$-linearized, and the analogous property holds for the dual; thus we can define the following.

\begin{definition}
\label{eq_Pic}
The \emph{equivariant Picard group} $\Pic_G(X)$ of a $G$-variety $X$ is the abelian group of $G$-linearized line bundles up to isomorphism; it comes with a homomorphism
\begin{align}
    \label{forget_lin}
    \phi \colon \Pic_G(X) \longrightarrow \Pic(X)
\end{align}
which forgets the linearization.
\end{definition}

\begin{lemma}
    \label{ker_finite}
    Assume $G$ is finite (not necessarily diagonalizable).
    Then the kernel of $\phi$ is of $\vert G \vert$-torsion.
\end{lemma}

\begin{proof}
The kernel of $\phi$ can be seen as all the $G$-linearizations of the trivial line bundle on $X$ up to isomorphism. Since all regular functions on $X$ are constant, such linearizations are Zariski-locally all of the form
\begin{align}
\label{action}
G \curvearrowright X \times \AA^1, \quad g \cdot (x,z) = (g\cdot x, f(g)z)
\end{align}
where $f$ is a character of $G$. Now, the group $G$ is obtained as an extension of $\pi_0(G)$ by $G^0$, as in (\ref{etale_connected}). Since $G^0$ is infinitesimal, the ring $\mathcal{O}(G^0)$ is local and we can write its invertible elements as  $k^\times+\mathfrak{m}$, where $\mathfrak{m}$ is the maximal ideal, satisfying
\[
h^{\vert G^0 \vert} = 0 \quad \text{for all } h \in \mathfrak{m}.
\]
This yields that $f^{\vert G^0\vert} \in X^*(\pi_0(G))$. 
Next, the characters of $\pi_0(G)$ form a finite abelian group of $\vert \pi_0(G) \vert$-torsion, and hence we can conclude that
\[
f^{\vert G \vert} = (f^{\vert G^0\vert })^{\vert \pi_0(G)\vert} = 1
\]
and we are done.
\end{proof}

\begin{lemma}
Let $\rho \colon T \longrightarrow W:= T/G$ be either a $G$-torsor, or assume that $W$ is a smooth curve. If $G$ acts generically freely on $T$, then the kernel and cokernel of 
\[
\rho^* \colon \Pic(W) \longrightarrow \Pic(T)
\]
are of $\vert G \vert$-torsion.
\end{lemma}

\begin{proof}
    In both assumptions, the $\mathcal{O}_W$-module $\rho_*\mathcal{O}_T$ is finite and locally free of rank $\vert G \vert$. This guarantees the existence of a canonical \emph{norm} homomorphism of abelian groups
    \[
    N_\rho \colon \Pic(T) \longrightarrow \Pic(W)
    \]
    such that $N_\rho (\rho^*(L)) \simeq L^{\otimes \vert G \vert}$
    for all invertible $\mathcal{O}_W$-modules. From this we can conclude that the cokernel of $\rho^*$ is of $\vert G \vert$-torsion.
    
    Moving on to the kernel, let us consider the group $\Pic_G(T)$ defined in \Cref{eq_Pic}. 
    Next we see that $\rho^*$ factorizes through $\Pic_G(T)$ and we get the following commutative diagram
    \[
    \begin{tikzcd}
        \Pic(W) \arrow[rr,"\rho^*"]\arrow[rd,"\rho^*_G",swap] && \Pic(T)\\
        & \Pic_G(T)\arrow[ru,"\phi",swap] &.
    \end{tikzcd}
    \]
    By \Cref{ker_finite}, it suffices to see that the induced morphism $\rho_G^*\colon \Pic(W)\to \Pic_G(T)$ is injective.  Let $L$ be a line bundle over $W$ such that $\rho_G^*(L)=\mathcal{O}_T$ equipped with the trivial action of $G$. Then $\rho^*(L)=\mathcal{O}_T$ and applying $\rho_*$, we obtain by the projection formula
    \[
     \rho_*(\mathcal{O}_T)= \rho_* \rho^*(L) =L \otimes \rho_*(\mathcal{O}_T),
    \]
    as $G$-linearized sheaves. Taking the $G$-invariants, we get that $L=\mathcal{O}_W$. 
\end{proof}

\begin{remark}
\label{remark picg}
Notice that, in the case where $\rho$ is a $G$-torsor, the map $\rho_G^\ast$ is an isomorphism. This is because a $G$-linearization of a line bundle on $T$ is exactly a descent datum for the finite and faithfully flat map $\rho \colon T \to T/G$. However, in general $\rho_G^\ast$ is not surjective, even when the target $W$ is a smooth curve. To see this, let $T$ be an elliptic curve and $G = \Z/2\Z$ acting by multiplication by $\pm 1$ on $T$; then the quotient map $\rho$ is a ramified covering of the projective line $W = \proj^1$. In this example, the invertible sheaf $\mathcal{O}_T(t)$, where $t \in T[2]$ is a ramification point, is $G$-linearized but is not in the image of $\rho^\ast$.
\end{remark}

    \begin{remark}
    More precisely, we have just showed that for a line bundle $L$ on $W$ such that $\rho^*(L) = \mathcal{O}_T$, we have a $G$-invariant section 
    \[
    s \colon \mathcal{O}_T \stackrel{\sim}{\longrightarrow} \rho^*(L)^{\otimes \vert G \vert}.
    \]
    Taking the pushforward by $\rho$ and taking invariants on both sides yields a section
    \[
    s^\prime = \rho_*^G s \colon \mathcal{O}_W \stackrel{\sim}{\longrightarrow} L^{\otimes \vert G \vert}
    \]
    trivialising the $\vert G \vert$-th power of $L$ as wanted.
    \end{remark}

\begin{corollary}
The kernel and cokernel of 
\[
q^* \colon \Pic (S) \longrightarrow \Pic(E \times X) \quad \text{and} \quad \pi^* \colon \Pic(Y) \longrightarrow \Pic(X)
\]
are of $\vert G \vert$-torsion.
\end{corollary}

Let us mention the following result, which seems to be a folklore statement known by experts but for which the authors could not find any explicit reference in the literature.

\begin{lemma}
        The functor $\underline{\Pic}^0$ is multiplicative on projective varieties. More precisely, let $Z$ and $W$ be projective varieties, then $\underline{\Pic}_{Z \times W}^0 = \underline{\Pic}_Z^0\times \underline{\Pic}_W^0$.
\end{lemma}

\begin{proof}
    Let us fix $z \in Z(k)$ and $w \in W(k)$; this gives the following two natural maps between the Picard schemes
    \[
    \alpha \colon \underline{\Pic}_Z \times \underline{\Pic}_W \longrightarrow \underline{\Pic}_{Z \times W}, \quad (L,M) \longmapsto pr_Z^*L \otimes pr_W^* M
    \]
    where $pr_Z$ and $pr_W$ denote the two projections, and
    \[
    \beta \colon \underline{\Pic}_{Z \times W} \longrightarrow \underline{\Pic}_Z \times \underline{\Pic}_W, \quad N \longmapsto (N_{\vert Z \times \{w\}}, N_{\vert \{z\} \times W}).
    \]
    By construction, the composite map $\beta \circ \alpha$ is the identity, so in particular $\ker(\alpha)$ is trivial. Moreover, the induced map of $\alpha$ on the Lie algebras is 
    \[
    d\alpha = H^1(Z,\mathcal{O}_Z) \oplus H^1(W,\mathcal{O}_W) \longrightarrow H^1(Z\times W, \mathcal{O}_{Z \times W}).
    \]
    Since both varieties are projective by assumption, it follows by the Künneth formula that $d\alpha$ is an isomorphism. The injectivity of $\alpha$ implies that $\underline{\Pic}_{Z\times W}$ is the product of $\mathrm{im}(\alpha)$ and of $\ker (\beta)$. Since $d\alpha$ is an isomorphism, we have moreover that $\ker(\beta)$ is a constant group, so its neutral component is trivial. In conclusion, $\alpha$ induces an isomorphism between the neutral components which are given by the functors $\underline{\Pic}^0$, as wanted.
\end{proof}

\begin{corollary}
\label{rem:picmult}
 The equality \[
    \underline{\Pic}^0_{E\times X} = E \times \underline{\Pic}^0_X\] holds. Moreover, by \Cref{Liu_unipotent}, the right hand side is the extension of an abelian variety by a unipotent group.
\end{corollary}

In order to compute more explicitly the Picard rank of $S$, we make use of the notion of Albanese variety.

\begin{definition}
    Let $Z$ be a projective variety and let us fix a base point $z_0 \in Z(k)$. Its \emph{Albanese variety} is the abelian variety $\mathrm{Alb}(Z)$, equipped with a morphism
    \[
    \mathrm{alb}_Z \colon Z \longrightarrow \mathrm{Alb}(Z), \quad z_0 \longmapsto 0
    \]
    satisfying the following universal property. For any abelian variety $A$ and any morphism $g \colon Z \rightarrow A$ sending $z_0$ to the neutral element of $A$, there is a unique factorisation as
    \begin{center}
        \begin{tikzcd}
            Z \arrow[d, "g"]  \arrow[rr, "\mathrm{alb}_Z" ]&& \mathrm{Alb}(Z)\arrow[dll, "\tilde{g}"]\\
            A && 
        \end{tikzcd}
    \end{center}
    where $\tilde{g}$ is a group homomorphism.
\end{definition}

%\begin{lemma}
%\label{lem:finiterank}
%    The rank of the free abelian group $\Hom_{gp} (\mathrm{Alb}(Z),E)$    is finite.
%\end{lemma}

%\begin{proof}
   % The category of abelian varieties up to isogeny is semisimple: in other words, the decomposition (up to isogeny) into a product of simple abelian varieties is unique (up to isogeny). We denote as $\, \sim \,$ the relation of being isomorphic up to isogeny. Going back to the projective variety $Z$, there is a unique decomposition
   % \[
   % \mathrm{Alb}(Z) \simeq \prod_{i=1}^r A_i^{n_i}
   % \]
   % where $n_i$ are positive integers and $A_i$ are simple abelian varieties, pairwise non isogenous to each other. Next, since $E$ is an elliptic curve, we have two possible cases. Either there is some $i$ such that $E \sim A_i$, in which case 
    %    \[
     %   \rank \Hom_{gp} (\mathrm{Alb}(Z),E) = n_i \cdot \rank \Hom_{gp} (E,E);
      %  \]
%or there is no such $i$, in which case the above rank is $0$. In both situations, we end up with a finite rank.
%\end{proof}

\begin{proposition}\label{picardrank}
    The Néron-Severi groups of $S$ and of $E \times X$ have same rank, equal to
    \[
    \rho(E \times X) = 2+\rank \Hom_{gp} (\mathrm{Alb}(X),E)
    \]
    For instance, if $X$ is rational then $S$ has Picard rank two, and the two contraction morphisms $f,h$ are respectively with targets $Y= \proj^1$ and $E/G$.
\end{proposition}

\begin{proof}
    By \Cref{Pic_curve}, the Picard functor $\underline{\Pic}_E$ is representable. The projection $pr_X\colon E\times X\to X$ induces an inclusion 
    \[
    pr_X^*\colon \Pic(X)\longhookrightarrow \Pic(E\times X),\]
    which comes with a natural section given by the restriction over $\{0_E\}\times X$. This yields
    \begin{align}
        \label{splitting}
        \Pic(E \times X) = \Pic(X) \oplus \underline{\Pic}_E(X).
    \end{align}
    By \Cref{Pic_curve} again, together with the fact that $E$ is an elliptic curve, we have the following short exact sequence of group schemes:
    \[
    0 \longrightarrow \underline{\Pic}_E^0 \longrightarrow \underline{\Pic}_E \longrightarrow \mathrm{NS}(E)\longrightarrow 0,
    \]
    where the map from the Picard scheme of $E$ to the constant scheme $\mathrm{NS}(E)=\underline{\Z}$ is given by the degree. 
    The following morphism
    \[
    \underline{\Z} \longrightarrow \underline{\Pic}_E, \quad m \longmapsto \mathcal{O}_E(m\cdot 0_E)
    \]
    is a section of the degree map. Moreover, there is a natural isomorphism $E \simeq \underline{\Pic}^0_E$, sending a point $x \in E$ to $\mathcal{O}_E(x-0_E)$. Hence applying the Picard functor of $E$ to $X$ we get
    \[
        \Pic(E\times X) = \Pic(X) \oplus \underline{\Pic}_E(X) \simeq  \Pic(X) \oplus (E \times \underline{\Z}) (X) = \Pic(X) \oplus \mathrm{Hom}(X,E) \oplus \Z.
    \]
    Moreover, for every $x_0 \in X(k)$, recall that we have 
    \[
    \mathrm{Hom}(X,E) = E\oplus \mathrm{Hom}(X,E;x_0\mapsto 0_E),
    \]
    where $\mathrm{Hom}(X,E;x_0\mapsto 0_E)$ is the group of morphisms from $X$ to $E$ sending $x_0$ to the neutral element $0_E$. By the rigidity lemma of abelian varieties, the latter group equals $\mathrm{Hom}_{gp}(\mathrm{Alb}(X),E)$ which is of finite rank by \cite[IV.19, Theorem 3]{mumford1974abelian}. We conclude by taking the quotient by $\Pic^0$ and then taking the rank, on both sides of the equality.
\end{proof}

The next natural question arising from the study of the Picard group of $S$ would be to understand the relative Picard functor of $S/Y$. 
A general result by Deligne, illustrated in \cite[Remark 5.27]{Kleiman}, is that the relative Picard functor $\Pic^0_{S/Y}$ is represented by a scheme if the morphism $f$ is proper and flat, with geometric fibers which are connected and reduced. In our situation, these assumption are rarely satisfied because the elliptic fibrations often have nonreduced fibers. The results collected in \cite[Théorème 1.5.1]{RaynaudPicard} give an answer to the representability in more general situations. In our context, $S/Y$ is flat, projective and of finite presentation, and its fibers are all irreducible varieties: by the theorem of J.P. Murre, this implies that the relative Picard functor of $S/Y$ is an algebraic group over $k$, that is possibly non-reduced. However, its reduced structure can be easily described, as we show in the following lemma
%\textcolor{teal}{Artin :
%\href{http://www.numdam.org/item/PMIHES_1970__38__27_0.pdf}{link}}

%\begin{proof}
%        Let us consider the relative Picard sheaf $\underline{\Pic}_{S/Y}$, which is represented by an algebraic group which is not necessarily reduced. Thus, we consider the base change by the morphism $\pi \colon X \rightarrow Y$, so that we are led to work with the projection $E \times X \to X$. By \Cref{rem:picmult}, we obtain that $\underline{\Pic}^0_{E \times X/X} \simeq E$. This implies that $\underline{\Pic}^0_{S/Y}$ is indeed representable and that it is isomorphic to $E$.\\
%\end{proof}

\begin{lemma}
\label{pic_fibers}
   Let $y \in Y \setminus Y_{\text{fr}}$ be a branch point and let $f^{-1}(y) \subset S$ be the corresponding fiber. Then the Picard group $\Pic^0(f^{-1}(y))$, seen as an abstract group, is isomorphic to $E(k)$. 
\end{lemma}

In particular, this excludes the possibility of a multiplicative or additive reduction, i.e.\ some fiber over a branch point $y \in Y\setminus Y_{\text{fr}}$ being isomorphic to $\Gm$ or to $\Ga$ respectively, as illustrated in the survey paper \cite[4.2]{Schuett} on elliptic surfaces with sections.
   
   \begin{proof}
   We make use of the local structure of \Cref{nowild}. Let $x \in X$ such that $\pi(x)=y$ and let $H=\boldsymbol{\mu}_n$ be the stabilizer of the $G$-action at the point $x$. Then the fiber of $f$ above $y$, as described in (\ref{fiber_contracted}), satisfies
    \begin{align}
    \label{F_fiber}
    f^{-1}(y) = E \times^H F, \quad \text{where} \quad F = \pi^{-1}(y) \simeq \Spec\left( k[T]/(T^n)\right).
    \end{align}
    In particular, the variable $T$ has weight equal to $1 \in X^*(H)$. 
    Let $pr_E\colon E \times F \rightarrow E$ denote the projection on $E$. Since $\mathcal{O}(F)$ is a local ring, its invertible elements are those in $k^\times+\mathfrak{m}$, where $\mathfrak{m}$ is generated by $T$. In particular, this yields
    \[
    (pr_E)_* \mathcal{O}_{E \times F}^\times = \left( \mathcal{O}_E[T]/(T^n) \right)^\times \simeq \mathcal{O}_E^\times \oplus \mathcal{O}_E^{n-1}.
    \]
    Next, we claim that
    \[
    R^1(pr_E)_* \mathcal{O}^\times_{E \times F} =0.
    \]
Taking higher direct images commutes with taking stalks (see \cite[Exposé VIII, Theorem 5.2]{SGA4}) and over any point of $E$, the fiber is isomorphic to the local scheme $F$. Hence, we can use that $H^1(F, \Gm)=0$ to conclude that the above direct image is trivial. 
When putting together the two equalities that we just proved, we get
    \begin{align*}
\Pic(E \times F) & = H^1(E \times F,\mathcal{O}_{E \times F}^\times) = H^1(E, \mathcal{O}_E^\times \oplus \mathcal{O}_E^{n-1})\\
& = \Pic(E) \times H^1(E,\mathcal{O}_E)^{n-1} = \Pic(E) \times k^{n-1}.
\end{align*}
Now, in order to get back to the Picard group of $f^{-1}(y)$, let us consider the map
\[
\Pic(f^{-1}(y)) = \Pic(E \times^H F) = \Pic_H(E \times F) \longrightarrow \Pic(E \times F).
\]
The second equality is due to the fact that the map $E \times F \to E\times^H F$ is an $H$-torsor (see \Cref{remark picg}). 
The image of the above map is given by $\Pic(E)$, because the $H$-invariant part of $\mathcal{O}(F)$ is just $k$, since $T,\ldots, T^{n-1}$ have all nontrivial weights. In particular, considering the degree $0$ part, we get that it is given simply by the group $E$.
\end{proof}

\subsection{The unipotent case: an example} 
At the moment we are unable to deal with more general classes of groups, such as linearly reductive or unipotent ones. This is due to the fact that the methods used for diagonalizable groups clearly do not apply. We set up here some notation and treat one example, and leave the more general case as an open question. Similar problems, involving wild $\Z/p\Z$ singularities, have been recently dealt with in works such as \cite{Lorenzini,Kentaro,stefan}.

\medskip

Let us consider a finite and constant group $G$ of arbitrary order. Let $L$ and $K$ be the function field of $X$ and $Y$ respectively, so that $G$ is the Galois group of the field extension $L/K$. Let $x \in X$ be some non-free point and $y \in Y$ its image. Then the local rings $\mathcal{O}_{X,x}$ and $\mathcal{O}_{Y,y}$ are discrete valuation rings; let $v(x)$ be the valuation at $x$. If $L_x$ and $K_y$ denote the respective completions of $L$ and $K$, then
\[
H(x) \defeq \Stab_G(x) = \Gal (L_x/K_y).
\]
Let $t$ be a local uniformizer at $x$, and let us consider the following integer (coming from the so-called Artin representation of $G$):
\[
a(x) = [G \colon H(x)] \sum_{g \in H(x)\setminus \{1\}} i_x(g), \quad \text{where} \quad i_x(g) := v(x)(g\cdot t -t).
\]
Following \cite[Chapter VI, Proposition 7]{Serre}, the \emph{Hurwitz formula} becomes
\begin{align}
\label{Hurwitz_constant}
\deg \omega_X = \vert G \vert \cdot \deg \omega_Y + \sum_{x \in X \setminus X_{\text{fr}}} a(x).
\end{align}
Unlike in the diagonalizable case, the ramification factors $a(x)$ can be \emph{strictly bigger} than $\vert H(x) \vert -1$; for instance in \Cref{ex_u} below, where a factor $2$ appears. 

\begin{example}
\label{ex_u}
    Let us assume that the characteristic is any prime number $p>0$ and consider the action of $G= \Z/p\Z$ on $\proj^1$ given on the local chart $\AA^1$ with coordinate $t$ by
    \[
    g\cdot t = t+g, \quad \text{for } g \in \Z/p\Z,
    \]
    so that the only fixed point is $x= \infty$ with local uniformizer $u = 1/t$. By the Hasse-Arf theorem, there exists a unique integer $m= m_0$ such that
    \[
    v(x)(g\cdot t -t) = 1+m_0 \text{ for all } g \neq 0.
    \]
    Let us compute it for $g=1$: locally around $x$ we have
    \[
    1\cdot u = \frac{1}{1+t} = u \frac{1}{1+u} = u \sum_i (-u)^i = u-u^2+u^3 +\ldots
    \]
In particular any $g\in G$ distinct from $0$ satisfies $i_x(g) = 2$, and hence we get $m_0 = 1$ and 
\[
a(x) = 2(p-1).
\]
Let us notice that this corresponds exactly to the realization of the surface $\mathsf{A}_0$ for an ordinary elliptic curve $E$, which is described in \Cref{A0}.
\end{example}

%-----------------------------------------------------------------------------

\subsection{Proofs of \Cref{theoremC} and \Cref{PropositionB}}\label{Section:proofs}

\begin{proof}[Proof of \Cref{theoremC}]
    By \Cref{trivialR1}, the elliptic fibration $f$ has only tame fibers. The equality $\kappa(S) = \kappa'(X)$ follows from (\ref{equalityKodaira}), the computations of $\chi(\mathcal{O}_S)$ and the irregularity $q(S)$ are done in \Cref{Irregularity}. The neutral component of the Picard group (seen as an abstract group) of each fiber $f^{-1}(y)$ is isomorphic to $E$ thanks to \Cref{pic_fibers}, while the fact that the Picard scheme of $S$ is reduced is shown \Cref{Pic0reduced}. Finally, the Picard rank of $S$ is computed in \Cref{picardrank}.
\end{proof}

\begin{proof}[Proof of \Cref{PropositionB}] The Kodaira dimension is determined by the degree of the dualizing sheaf, as illustrated in \Cref{cor:degreeomega}. The different cases are analyzed in \Cref{sec:kodaira}, the case where $Y$ is the projective line being formulated in \Cref{Nfibers}.
\end{proof}

\printbibliography

\end{document}